\documentclass[12pt,oneside,reqno]{amsart}

\usepackage[T1]{fontenc}
\usepackage[utf8]{inputenc}
\usepackage{amsmath}
\usepackage{lipsum}
\usepackage{amssymb}
\usepackage{amsfonts}
\usepackage{enumitem}
\usepackage{bbm}
\usepackage{hyperref}
\usepackage{ragged2e}
\hypersetup{colorlinks=true}

\usepackage[a4paper, left=2.5cm, right=2.5cm, top=2.5cm, bottom=2.5cm]{geometry}
\usepackage{amsthm}
\usepackage{indentfirst}
\DeclareMathAlphabet{\mathcal}{OMS}{zplm}{m}{n}

\newcommand{\tmod}{\mathrm{mod}}
\newcommand{\R}{\mathbb{R}}
\newcommand{\bC}{\mathbb{C}}

\newcommand{\N}{\mathbb{N}}

\newcommand{\Z}{\mathbb{Z}}

\newcommand{\cH}{\mathcal{H}}
\newcommand{\cL}{\mathcal{L}}

\newcommand{\cG}{\mathcal{G}}

\newcommand{\cT}{\mathcal{T}}

\newcommand{\cI}{\mathcal{I}}
\newcommand{\cY}{\mathcal{Y}}

\newcommand{\sfT}{\mathsf{T}}
\newcommand{\sfN}{\mathsf{N}}

\newcommand{\tN}{\mathrm{N}}

\newcommand{\bfd}{\mathbf{d}}
\newcommand{\bfx}{\mathbf{x}}
\newcommand{\bfy}{\mathbf{y}}
\newcommand{\bfz}{\mathbf{z}}

\newcommand{\bfu}{\mathbf{u}}
\newcommand{\bfv}{\mathbf{v}}
\newcommand{\bfp}{\mathbf{p}}
\newcommand{\bfw}{\mathbf{w}}

\newcommand{\supp}{\mathrm{supp}}

\newcommand{\dimH}{\dim_{\mathrm{H}}}
\newcommand{\lodimA}{\underline\dim_{\mathrm{A}}}
\newcommand{\BAD}{\mathrm{BAD}}

\newcommand{\id}{\mathrm{id}}
\newcommand{\dist}{\mathrm{dist}}

\newcommand{\linte}[1]{{\left\lfloor #1 
		\right\rfloor}}
\newcommand{\uinte}[1]{{\left\lceil #1 
		\right\rceil}}

\theoremstyle{plain}

\newtheorem{theorem}{Theorem}
\newtheorem{lemma}{Lemma}
\newtheorem{proposition}{Proposition}

\newtheorem{corollary}{Corollary}
\newtheorem{claim}{Claim}

\newtheorem*{claim*}{Claim}

\newtheorem*{mah}{{\rm \textbf{Mahler's Problem}}}

\theoremstyle{definition}
\newtheorem{definition}{Definition}
\newtheorem{example}{Example}

\newtheorem*{examples*}{Examples}
\newtheorem*{example*}{Example}
\newtheorem{remark}{Remark}

\newtheorem{question}{Question}

\newtheorem*{notations*}{Notations}
\newtheorem*{notation*}{Notation}

\numberwithin{equation}{section}
\numberwithin{question}{section}
\numberwithin{theorem}{section}
\numberwithin{thm}{section}
\numberwithin{lemma}{section}
\numberwithin{proposition}{section}
\numberwithin{cor}{section}
\numberwithin{corollary}{section}
\numberwithin{claim}{section}
\numberwithin{definition}{section}
\numberwithin{example}{section}
\numberwithin{remark}{section}
\numberwithin{notations}{section}
\numberwithin{notation}{section}
\numberwithin{claim}{section}
\numberwithin{problem}{section}
\numberwithin{figure}{section}

\allowdisplaybreaks
\title{A winning approach to the intersections of twisted non-recurrent sets with fractals}
\author[J. Huang]{Junjie Huang}
\address[J. Huang]{School of Mathematics, South China University of Technology, Wushan Road 381, Tianhe District, Guangzhou, China}
\email{h1135778478@163.com}

\author[B. Li]{Bing Li}
\address[B. Li]{School of Mathematics, South China University of Technology,  Wushan Road 381, Tianhe District, Guangzhou, China}
\email{scbingli@scut.edu.cn}

\author[B. Wang]{Bo Wang}
\address[B. Wang]{School of Mathematics, Sun Yat-sen University, No.135 Xingang West Road, Haizhu District, Guangzhou, China; School of Mathematics, Jiaying University, No. 100 Meisong Road, Meijiang District, Meizhou, China}
\email{math\_bocomeon@163.com}
\author[N. Yuan]{Na Yuan}
\address[N. Yuan]{School of Mathematics, Guangdong University of Education,  Xingang Middle Road 351, Guangzhou,  China}
\email{yuanna199@gmail.com}

\begin{document}
\begin{abstract}
    In this paper, we prove that if $S\subseteq\R^d$ is hyperplane absolute winning on a closed hyperplane diffuse set $L\subseteq\R^d$, then $\dimH S\cap K=\dimH K$ for any irreducible self-conformal set $K\subseteq L$ without assuming any separation condition on $K$.  The result is then applied to obtain the Hausdorff dimension of intersections between irreducible self-conformal sets and twisted non-recurrent sets $\tN(T,\cG)$ defined as
    \begin{equation*}
        \tN(T,\cG):=\left\{\bfx\in[0,1]^d:\liminf_{n\to\infty}\|T^n(\bfx)-g_n(\bfx)\|>0\right\},
    \end{equation*}
    where $T:[0,1]^d\to[0,1]^d$ belongs to a broad class of product maps, $\cG:=\{g_n\}_{n\in\N}$ is a sequence of self-maps on $[0,1]^d$ with uniform Lipschitz constant and $\|\cdot\|$ denotes the maximal norm in $\R^d$. {When $T$ is the $\beta$-transformation on $[0,1]$, it provides a positive answer to a question raised informally by Broderick, Bugeaud, Fishman, Kleinbock and Weiss (Math. Res. Lett., 2010). For the case $T$ is a $d\times d$ diagonal matrix transformations, our results provide a partial answer asked in a paper of Li, Liao, Velani and Zorin (Adv. Math., 2023). A natural generalization to non-autonomous setting is also obtained.}
\end{abstract}
\maketitle 

% \section{Outline}
% \begin{itemize}
%     \item Main question: Let $K\subseteq[0,1]^d$. What is the Hausdorff dimension of $N(T,\cG)\cap K$?\medskip

%     \item Mahler question: from the viewpoint of Badly approximable numbers in fractals.
%     \medskip

%     \item $BAD=N(T,0)$, where $T:[0,1]\to[0,1]$ is a Gauss map.
%     \medskip

%     \item How to introduce the non-autonomous dynamics?
% \end{itemize}

\section{Introduction}

\subsection{Background and motivation} 
In 1984, Mahler \cite{MR732177} posed a number of fundamental problems, among which the following has played a pivotal role in  motivating subsequent research in Diophantine approximation on fractals:
\begin{mah}
How close can irrational elements of Cantor's set be approximated by rational numbers
    \begin{itemize}
        \item[(i)] in Cantor's set, and
        \medskip
        \item[(ii)] by rational numbers not in Cantor's set?
    \end{itemize}
\end{mah}
\noindent Up to now, the problem has attracted considerable attention from numerous researchers and led to a great deal of advances in the study of distribution of rational numbers inside or near fractal sets; see for example \cite{MR4525557,MR4519189,MR4705765,MR4905584,benard2024,benard2025,MR3500835,chow2024,MR3237740,MR3350721,MR4574663,MR2295506,MR4875365,MR4240610,MR4670093,MR3673046,MR1875309,yu2021} and the references therein.

 A particularly active line of research  inspired by Mahler's Problem concerns the Hausdorff dimension of intersection of the badly approximable set with fractals. To state it precisely,  denote
\begin{equation*}
   \BAD_d:=\left\{\bfx\in[0,1]^d:\exists\, c(\bfx)>0\text{\ such that}\ \left\|\bfx-\frac{\bfp}{q}\right\|>\frac{c(\bfx)}{q^{1+1/d}}\,,\ \forall\ (\bfp,q)\in\Z^d\times\N\right\}.
\end{equation*}
 Here,  $\|\cdot\|$ is the maximal norm in $\R^d$ defined by
\begin{equation}\label{maxnorm}
    \|\bfx\|:=\max_{1\leq i\leq d}|x_i| \qquad \forall\ \bfx=(x_1,x_2,...,x_d)\in\R^d
\end{equation}
and $|x|$ is the absolute value of $x\in\R$.
The elements in $\BAD_d$ are referred to as \emph{badly approximable vectors} and the set itself is called the \emph{badly approximable set}. It is well-known (e.g.,   Chapter 3 in \cite{MR568710}) that  although the set $\BAD_d$ has zero $d$-dimensional Lebesgue measure, its  Hausdorff dimension is full, namely equal to $d$. We refer to  Chapter 2 of \cite{MR2118797} for a comprehensive introduction to the Hausdorff dimension (denoted by $\dimH$). This motivates the natural and extensively studied question of whether $\BAD_d$ has full Hausdorff dimension within various kinds of fractal sets. Over the past decades, this line of inquiry has produced a large number of affirmative results, especially for self-conformal sets and self-affine sets satisfying certain non-concentration assumptions and separation conditions; see \cite{MR2644378,MR2818688,MR3826829,MR3904182,MR2191212,MR2581371} and the references therein. In terms of self-conformal sets, it has been shown in the work of Das, Fishman, Simmons and Urba\'{n}ski \cite{MR3826829} that if $K\subseteq[0,1]^d$ is an irreducible self-conformal set that satisfies the open set condition (OSC),  then  $\dimH \BAD_d\cap K=\dimH K$. Here, we say that a self-conformal set in $\R^d$ is \emph{irreducible} if it is not contained in a $(d-1)$-dimensional real-analytic submanifold of $\R^d$ and we refer to Section \ref{defofselfcon} for detailed information about self-conformal sets.  A natural and meaningful step extending the result is to determine whether the open set condition can be removed. This brings about the following question.
\begin{question}\label{removeosc}
    Let $K\subseteq[0,1]^d$ be an irreducible self-conformal set.  Is it true that $$\dimH\BAD_d\cap K=\dimH K\,?$$
\end{question}
\noindent This paper provides a  positive answer as discussed in Remark \ref{answerremosc} of Section \ref{HYWINN}.  The approach is closely linked to the notion of Schmidt's game and its variants, originating from the influential work of Schmidt \cite{MR195595}.  
% We refer to the discussion below and to Section \ref{HYWINN} for detailed account of Schmidt's game and its significant role in the intersection with fractals.
Here, we recall a key property that if $E\subseteq\R^d$ is $\alpha$-winning  for some $\alpha\in(0,1)$, then $\dimH E=d$; 
 see Section \ref{HYWINN} for detailed information.

  We now introduce a generalization of the badly approximable set. When $d=1$,  it is well known (Proposition 3.10 of \cite{MR2723325}) that
\begin{equation}\label{badtnx}
    \BAD_1=\left\{x\in[0,1]:\liminf_{n\to\infty}T_G^n(x)>0\right\},
\end{equation}
where $T_G:[0,1]\to[0,1]$ is the Gauss map defined by
\begin{eqnarray}\label{defofgauss}
    T_G(x):=\left\{
\begin{aligned}
    &1/x-\linte{1/x},\ \text{if $x\in(0,1]$},\\
    &0,\qquad \quad \ \ \quad \ \,   \text{if $x=0$}
\end{aligned}
    \right.
\end{eqnarray}
and $\linte{x}:=\max\{z\in\Z:z\leq x\}$ for any $x\in\R$.
 This motivates a dynamical analogue of the badly approximable set: for a metric space $(X,\bfd)$ and a map $T:X\to X$, define 
\begin{equation*}
    \BAD(T,y):=\left\{x\in X:\liminf_{n\to\infty}\bfd(T^nx,y)>0\right\}.
\end{equation*}
Roughly speaking, $\BAD(T,y)$ consists of those points whose orbits stay away from some ``avoidance'' ball centered at $y$. Specifically,  $x\in\BAD(T,y)$ if and only if there exists $\delta=\delta(x)>0$ such that $T^nx\notin B(y,\delta)$ for any sufficiently large $n\in\N$. Here, $B(y,r)$ denotes the open ball with radius $r$ centered at $y$.  

 In the case $X=[0,1]$, $T=T_G$ is the Gauss map defined in \eqref{defofgauss} and $y=0$, it follows directly from \eqref{badtnx} that the set $\BAD(T_G,0)$ coincides with $\BAD_1$ and thus $\dimH \BAD(T_G,0)=1$. This result was generalized by Schmidt \cite{MR195595} who  showed that it is in fact $1/2$-winning. Beyond this classical setting, a rich body of results exist concerning the winning property of $\BAD(T,y)$. For one-dimensional systems, Tseng \cite{MR2480100} proved that for any $C^2$ expanding self-map $T$ on circle $S^1$, the set $\BAD(T,y)$ is $1/2$-winning. Besides the smooth maps,  it has been established from multiple sources \cite{MR2739401,MR2660561,MR3206688,MR3021798} that if $T$ is a $\beta$-transformation, Gauss map or L\"{u}roth transformation on the unit interval, then $\BAD(T,y)$ is $1/2$-winning for any $y\in[0,1]$. More recently, Hu, Li and Yu \cite{hu2025} extended this result to a much broader setting by showing that for any piecewise $C^{1+\kappa}$ expanding map $T$ on $[0,1]$  with finite branches (see Definition \ref{defofpiee}), the set $\BAD(T,y)$ is $1/2$- winning for any $y\in[0,1]$. For non-expanding systems, Duvall \cite{MR4157863} showed that if $T$ belongs to a certain kind of nonuniformly expanding maps on $[0,1]$, then $\BAD(T,0)$ is $\alpha$-winning for some $\alpha\in(0,1)$. In terms of systems in high-dimensional space, we refer to \cite{MR4389795,MR3519428} and references therein for corresponding results on partially hyperbolic diffeomorphisms, and to \cite{MR2818688,MR980795} and the works cited there for results regarding  endomorphisms on high-dimensional torus.

From the viewpoint of Mahler's Problem, it is of particular interest in understanding the size, in the sense of Hausdorff dimension, of the set $\BAD(T,y)\cap K$ for various kinds of subsets $K\subseteq X$. Significant progress has been made under the assumption that $T$ is an integer matrix transformation and $K$ is the support of a measure with nice geometric properties. We say that a Borel probability measure $\mu$ on $\R^d$ is \emph{absolutely decaying} if there exist $C>0$, $r_0>0$ and $s>0$ such that
 \begin{equation*}
     \mu\big(B(\bfx,r)\cap(\cL)_{\gamma r}\big)\leq C\cdot\gamma^{s}\cdot\mu(B(\bfx,r))
 \end{equation*}
 for any $\bfx\in\supp(\mu)$, $0<r\leq r_0$, $0<\gamma\leq 1$ and affine hyperplane $\cL$ with dimension $d-1$. Here, $(E)_{\rho}$ denotes the $\rho$-neighborhood of $E$ defined as
  \begin{equation*}
      (E)_{\rho}:=\left\{\bfx\in\R^d:\exists\,\bfy\in E\ \text{such that}\ \|\bfx-\bfy\|<\rho\right\}.
  \end{equation*} Given $s>0$, the measure $\mu$ is called \emph{$s$-Ahlfors regular} if there exist $C>1$ and $r_0>0$ such that
 \begin{equation*}
     C^{-1}r^{s}\leq\mu(B(\bfx,r))\leq Cr^{s}
 \end{equation*}
 for any $\bfx\in\supp(\mu)$ and $0<r\leq r_0$. It is called \emph{Ahlfors regular} if it is $s$-Ahlfors regular for some $s>0$. It was shown in \cite{MR2818688} that if 
\begin{equation}\label{txeqaxmod1}
    T\bfx=A\bfx\ \ \tmod \ 1,\qquad\forall\ \bfx\in[0,1]^d,
\end{equation}
where $A$ is a $d\times d$ integer matrix with an eigenvalue of modulus greater than one,  then $\dimH\BAD(T,\bfy)\cap K=\dimH K$ for any $\bfy\in[0,1]^d$ provided that $K\subseteq[0,1]^d$ is the support of an absolutely decaying and Ahlfors-regular measure.   This result was further extended to sequences $\cY=\{\bfy_n\}_{n\in\N}$ in $[0,1]^d$ via the set  
\begin{equation*}
    \BAD(T,\cY):=\Big\{x\in [0,1]^d:\liminf_{n\to\infty}\|T^n{\bfx}-{\bfy}_n\|>0\Big\}
\end{equation*}
for which the same dimensional conclusion still holds.  Note that while $\mathrm{BAD}(T, y)$ involves a fixed “avoidance center” $y$, the set $\mathrm{BAD}(T, \mathcal{Y})$ allows the centers to vary with $n\in\N$, which introduces additional complexity. However, the above results fail to extract any effective information on the size of $\BAD(T,\mathcal{Y})\cap K$ when $A$ is a general $d\times d$ real matrix. It is even not known whether $\BAD(T,\mathcal{Y})$ has full Hausdorff dimension in this case.
%  In the  one dimensional case, these results apply to the $\times m$ map $T_m(x):=mx\ \, \tmod \ 1$
%  with integer $m\geq2$,  but fail to extend to the $\beta$-transformation $T_{\beta}(x):=\beta x\ \, \tmod \ 1$ 
% for non-integer $\beta>1$. 
The limitation stems from the fact that the proofs for the integer case rely heavily on the algebraic property  $T^n(x)=A^n \bfx\  \tmod \ 1$ for any $n\in\N$ and $\bfx\in[0,1]^d$ when $A$ is an integer matrix, which is absent in the setting of  non-integer matrix transformation. As a result, the following question remains unsolved in the existing literature.
\begin{question}\label{questbeta}
  Let $T:[0,1]^d\to[0,1]^d$ be a real matrix transformation as in \eqref{txeqaxmod1} whose corresponding matrix has an eigenvalue of modulus greater than one and let $\mathcal{Y}=\{\bfy_n\}_{n\in\N}$ be a sequence of points in $[0,1]^d$.
 \begin{itemize}
     \item[(i)] Is it true that $\dimH\BAD(T,\mathcal{Y})=d$?\medskip

    \item[(ii)] Let $K\subseteq[0,1]^d$ be a subset. What can we say about the Hausdorff dimension of the intersection $\BAD(T,\cY)\cap K$?
 \end{itemize}
\end{question}

 Recent progress on Schmidt's game does offer some insights in the one-dimensional case: a combination of    Corollary 5.4 of \cite{MR980795}  and Corollary 1.12 of the recent paper \cite{MR4833656} implies that if $E\subseteq[0,1]$ is $1/2$-winning on $[0,1]$ and $K\subseteq[0,1]$ is the support of an Ahlfors-regular measure, then $\dimH E\cap K=\dimH K$.
  In conjunction with the previously noted  result of Hu, Li and Yu \cite{hu2025}, this implies that the desired dimensional equality $\dimH\BAD(T,y)\cap K=\dimH K$ holds when $K\subseteq[0,1]$ is the support of an Ahlfors regular measure, $T:[0,1]\to[0,1]$ is a piecewise $C^{1+\kappa}$ expanding  map  with finite branches and  $y\in[0,1]$. Such result applies to the $\beta$-transformation $T_{\beta}(x):=\beta x\ \tmod\ 1$ on $[0,1]$ for any $\beta\in\R$ with $|\beta|>1$.  However, the proof  in \cite{hu2025} depends crucially on the fact that the avoidance center $y$ is fixed and does not adapt to the varying centers $\cY$ (see Remark \ref{onedimtherek} (i) for a more detailed discussion). Therefore, the dimension theory of the corresponding intersections in case of general sequences $\mathcal{Y}$ is still not fully understood in the literature.   As outlined above, clearing up the Question \ref{questbeta} is a natural and crucial step towards completing the picture of badly approximable theory in the context of non-integer dynamics.                  

 Next, we incorporate $\BAD(T,\cY)$ within the following more general framework -- the twisted non-recurrent set.
 \begin{definition}\label{defofnnontwi}
    Let $(X,\bfd)$ be a metric space, $T$ be a self-map on $X$ and $\cG:=\{g_n\}_{n\in\N}$ be a sequence of self-maps on $X$. Define the \emph{twisted non-recurrent set $\tN(T,\cG)$ associated with $T$ and $\cG$} as
    \begin{eqnarray*}
        \tN(T,\cG):=\left\{x\in X: \liminf_{n\to\infty}\bfd(T^nx,g_n(x))>0\right\}.
    \end{eqnarray*}
\end{definition}
\noindent  This is more general since $\tN(T,\cG)=\BAD(T,\cY)$ if $g_n(x)\equiv y_n$ for all $n\in\N$. By definition, $x\in\tN(T,\cG)$ if and only if there exists $\delta=\delta(x)>0$ such that $T^nx\notin B(g_n(x),\delta)$ eventually. In contrast to $\BAD(T,y)$ and $\BAD(T,\cY)$ where the centers of ``avoidance'' balls are fixed, the centers $g_n(x)$ here depend explicitly on the point $x\in X$. This intrinsic dependence introduces additional difficulty in analyzing the structure of $\tN(T,\cG)$.

A primary motivation of bringing in the set-up is to encompass a fundamental class of objects: the non-recurrent sets
\begin{eqnarray*}
    \tN(T):=\left\{x\in X:\liminf_{n\to\infty}\bfd(T^nx,x)>0\right\}.
\end{eqnarray*}
The classical Poincar\'{e}'s Recurrence Theorem states that $\mu(\tN(T))=0$ if the metric space $X$ is separable and the Borel probability measure $\mu$ on $X$ is $T$-invariant; that is to say that $\mu(T^{-1}E)=\mu(E)$ for any $\mu$-measurable subset $E\subseteq X$. It is therefore desirable and challenging to determine the Hausdorff dimension of $\tN(T)$. Previous work \cite{nayuan2021} established that $\mathrm{N}(T)$ has full Hausdorff dimension for any expanding Markov map $T:[0,1]\to[0,1]$, covering important examples such as $\beta$-transformations, the Gauss map, and the Lüroth map.  It was then generalized to Markov expanding maps on  subsets $X\subseteq[0,1]$ which are  closures of disjoint union of open intervals \cite{nayuan2021}. Regarding the intersection with fractals, the Theorem 1.3 in \cite{MR2818688} implies that $\dimH\tN(T)\cap K=\dimH K$ if $T:[0,1]^d\to[0,1]^d$ is a matrix transformation as in \eqref{txeqaxmod1}  with respect to a $d\times d$ integer matrix with an eigenvalue of modulus greater than one  and $K\subseteq[0,1]^d$ is the support of an absolutely decay and Ahlfors-regular measure. For the same reason demonstrated earlier, however, these results rely heavily on the algebraic structure of integer matrix transformations and do not extend to general real matrices. Even in the one-dimensional setting, while the desired dimensional equality holds for the $\times m$ map $T_m(x):=mx\ \tmod\ 1$ on $[0,1]$ with $m\in\Z$ and $|m|\geq2$, results in the existing works do not tell  whether this continues to hold for $\beta$-transformation $T_{\beta}(x):=\beta x\ \tmod \ 1$ with $\beta\in\R$ and $|\beta|>1$.
% , let alone for general piecewise  expanding maps (see Definition \ref{defofpiee}) and twisted non-recurrent sets. 

The above limitations highlight a substantial gap in the existing intersection theory. In light of this, we formulate the following broad generalization of Question \ref{questbeta} incorporating twisted non-recurrent sets.
\begin{question}\label{quesofrpietwi}
Let $T:[0,1]^d\to[0,1]^d$ be a real matrix transformation as in \eqref{txeqaxmod1} whose corresponding matrix has an eigenvalue of modulus greater than one. Let $\cG=\{g_n\}_{n\in\N}$ be a sequence of self-maps on $[0,1]^d$. 
 \begin{itemize}
     \item[(i)] Is it true that $\dimH\tN(T,\cG)=d$?\medskip

 \item[(ii)] Let $K$ be a subset of $[0,1]^d$. What can we say about the Hausdorff dimension of the intersection $\tN(T,\cG)\cap K$?
 \end{itemize}
\end{question}

  In the next part, we introduce the machinery of Schmidt games and their variants, which will provide a powerful framework for addressing Question \ref{quesofrpietwi}. A partially positive answer to the question will be presented in Section \ref{twishyperplane}.

\subsection{Intersections with fractals: a winning approach}\label{HYWINN}

Throughout, $\R^d$ is equipped with the maximal norm defined in \eqref{maxnorm}. We begin by describing the game introduced by Schmidt \cite{MR195595}. Let $\alpha,\gamma\in(0,1)$ and let $K\subseteq\R^d$ be a closed subset. The \emph{$(\alpha,\gamma)$-game on $K$} is played as follows:  two players, called Alice and Bob, take turns choosing nested closed balls centered in $K$ that satisfy the inclusions
\begin{equation*}
    B_1\supseteq A_1\supseteq B_2\supseteq A_2\supseteq\cdots\supseteq B_k\supseteq A_k\supseteq\cdots
\end{equation*}
and radius conditions
\begin{equation*}
    |A_k|=\alpha\cdot|B_k|,\quad|B_{k+1}|=\gamma\cdot|A_k|,\quad\forall \ k\in\N,
\end{equation*}
where $A_k$ and $B_k$ denote the $k$-th moves of Alice and Bob, respectively, and $|E|$ is the diameter of the set $E$. We say that a set $S\subseteq \R^d$ is \emph{$(\alpha,\gamma)$-winning on $K$} if Alice has a strategy of choosing balls to ensure that $\bigcap_{k=1}^{\infty}B_k\subseteq S$ whatever Bob chooses his balls. It is called \emph{$\alpha$-winning on $K$} if it is $(\alpha,\gamma)$-winning on $K$ for any $\gamma\in(0,1)$. We simply say that $S$ is \emph{winning on $K$} if it is $\alpha$-winning on $K$ for some $\alpha\in(0,1)$.  Some important properties of $\alpha$-winning sets are summarized as follows.
\begin{proposition}[{\cite{MR195595}}]\label{winnpro}
    Let $E$, $E_n$ be subsets of $\R^d$ for any $n\in\N$, let $K\subseteq\R^d$ be a closed subset with non-empty interior and let $\alpha\in(0,1)$.
    \begin{itemize}
        \item[(i)] If $E$ is $\alpha$-winning on $K$, then $\dimH E=d$.
 \medskip
 \item[(ii)] If $E$ is $\alpha$-winning on $K$, then it is $\alpha'$-winning on $K$ for any $0<\alpha'\leq\alpha$.
 \medskip

 \item[(iii)] If $E_n$ is $\alpha$-winning on $K$ for any $n\in\N$, then $\bigcap_{n=1}^{\infty}E_n$ is $\alpha$-winning on $K$.\medskip

 \item[(iv)] If $E\neq\R^d$ and $\alpha>1/2$, then $E$ is not $\alpha$-winning on $\R^d$.
    \end{itemize}
\end{proposition}

The Proposition \ref{winnpro} (i) illustrates the largeness of winning sets, in the sense of dimension, whose ambient spaces possess non-empty interior. It is therefore natural and desirable to expect that if $K\subseteq\R^d$ is a closed subset having positive dimension and $E\subseteq K$ is winning on $K$, then the Hausdorff dimension of $E$ is also large. The next result confirms the expectation in some sense. To state it precisely,  define the lower Assouad dimension of a given subset $E\subseteq\R^d$ as
\begin{equation*}
                \lodimA E:=\sup\left\{s\geq0\left|
                \begin{aligned}
                    &\exists\, c>0\text{ such that }\forall\,\bfx\in E,\,\forall\,\epsilon,\rho\in(0,1),\\
                &\text{it holds that } N_{\epsilon\rho}(B(\bfx,\rho)\cap E)\geq c\epsilon^{-s}
                \end{aligned}
                \right.\right\},
            \end{equation*}
where $N_r(E)$ denotes the smallest number of sets with diameter $r$ needed to cover $E$.
\begin{proposition}[{\cite[Proposition 2.5]{MR3904182}}]\label{dimhlowdima}
    Let $K\subseteq\R^d$ be a closed set and $E\subseteq K$ be winning on $K$, then
    \begin{equation*}
        \dimH E\geq\lodimA K.
    \end{equation*}
\end{proposition}

 In view of Question \ref{quesofrpietwi}, let $E=\tN(T,\cG)$ be the twisted non-recurrent set discussed in the previous subsection and let $K\subseteq\R^d$ be a closed subset.  In order to estimate $\dimH E\cap K$ via Proposition~\ref{dimhlowdima}, it is essential to prove that $E\cap K$ is winning on $K$. An effective way of establishing this is to use the  hyperplane absolute game, a variant of Schmidt's game introduced by Broderick, Fishman, Kleinbock, Reich and Weiss in the paper \cite{MR2981929}. The formal definition is as follows.

 Let $\gamma\in(0,1/3)$ and $K\subseteq\R^d$ be a closed subset. The \emph{$\gamma$-hyperplane absolute game  on $K$} is played as follows. There are two players in the game, named Bob and Alice. Initially, Bob chooses arbitrarily a closed ball $B_1$ centered in $K$. Then, in each stage, after Bob chooses $B_i$, Alice chooses $A_i=(\cL_i)_{\epsilon_i}$ for some $(d-1)$-dimensional affine hyperplane $\cL_i\subseteq\R^d$ and $0<\epsilon_i\leq\gamma\cdot|B_i|/2$. 
 Next, Bob chooses a closed ball $B_{i+1}$ centered in $K$ that satisfies
  \begin{equation*}
      B_{i+1}\subseteq B_{i}\setminus A_i,\qquad |B_{i+1}|\geq \gamma\cdot|B_i|.
  \end{equation*}
  A set $S\subseteq \R^d$ is said to be \emph{$\gamma$-hyperplane absolute winning on $K$} if Alice has a strategy to ensure that $\bigcap_{i=1}^{\infty}B_i\cap S\neq\emptyset$ whatever Bob chooses his balls. It is \emph{hyperplane absolute winning on $K$} if there exists $\gamma_0>0$ such that it is $\gamma$-hyperplane absolute winning on $K$ for any $0<\gamma\leq\gamma_0$. In the one dimension, the notion of hyperplane absolute winning coincides with that of absolute winning introduced by McMullen \cite{MR2720230}.

 Before exploring the properties of the hyperplane absolute game, we first introduce the useful notion of the hyperplane diffuse set. Given a set $E\subseteq\R^d$ and $\gamma>0$, we say that $E$ is \emph{$\gamma$-hyperplane diffuse} if there exists $r_0>0$ such that for any $\bfx\in E$ and affine hyperplane $\cL$ of dimension $d-1$, we have
 \begin{equation*}
\big(B(\bfx,r)\setminus(\cL)_{\gamma r}\big)\cap E\neq\emptyset,\qquad\forall \ 0<r\leq r_0.
 \end{equation*}
   We simply say that $E$ is \emph{hyperplane diffuse} if it is $\gamma$-hyperplane diffuse for some $\gamma>0$. For example, $[0,1]^d$ is hyperplane diffuse whereas any lower dimensional submanifold of $\R^d$ is not. 
  
  The hyperplane absolute winning property possesses the following basic properties.
  \begin{proposition}[{\cite{MR2981929}}]\label{proofhypdi}
      Let $K\subseteq L$ be two closed subsets in $\R^d$ and let $E\subseteq\R^d$. We obtain that
      \begin{itemize}
          \item[(i)] if $K$ is hyperplane diffuse and $E$ is hyperplane absolute winning on $K$, then $E$ is winning on $K$.
          \medskip

          \item[(ii)] if both $K$ and $L$ are hyperplane diffuse and $E$ is hyperplane absolute winning on $L$, then $E$ is hyperplane absolute winning on $K$.
          \medskip

          \item[(iii)] if there exists $\gamma_0$ such that for any $n\in\N$, $E_n$ is $\gamma$-hyperplane absolute winning on $K$ for any $0<\gamma\leq\gamma_0$, then $\bigcap_{n=1}^{\infty}E_n$ is $\gamma$-hyperplane absolute winning on $K$ for any $0<\gamma\leq\gamma_0$.
      \end{itemize}
  \end{proposition}
 \noindent The Proposition \ref{proofhypdi} (i) and (ii) illustrate the strength of the hyperplane absolute game: the property of being hyperplane absolute winning on the full space implies winning on a broad class of subsets. On combining this  with Proposition \ref{dimhlowdima}, we obtain the following corollary regarding the dimension of intersections.
 \begin{corollary}\label{corscplodimak}
     Let $L$ be a closed and hyperplane diffuse subset in $\R^d$ and let $S\subseteq\R^d$ be hyperplane winning on $L$. Then, for any closed hyperplane diffuse subset $K\subseteq L$, we obtain that $S$ is hyperplane absolute winning on $K$ and that $$\dimH S\cap K\geq\lodimA K.$$
 \end{corollary}

 We now demonstrate an application of  Corollary \ref{corscplodimak}. According to the Proposition 5.1 of \cite{MR2981929}, the support of an absolutely decaying measure is hyperplane diffuse. Furthermore, it is a standard fact that $\lodimA\supp(\mu)=\dimH\supp(\mu)=s$ for any $s$-Ahlfors regular measure $\mu$. Integrating these results with Corollary \ref{corscplodimak}, we obtain the following result on the full Hausdorff dimensional properties of hyperplane absolute winning sets restricted to a class of subsets.
 \begin{corollary}
     Let $L\subseteq\R^d$ be a closed hyperplane diffuse set and let $S\subseteq\R^d$ be hyperplane absolute winning on $L$. Then, for any absolutely decaying and Ahlfors regular measure $\mu$ supported on $L$, we have
     \begin{equation}\label{dimequhypwin}
         \dimH S\cap\supp(\mu)=\dimH\supp(\mu).
     \end{equation}
 \end{corollary}
 \begin{remark}\label{existresosc}
     The corollary applies, in particular, to a broad class of self-conformal sets. Precisely,  let $K$ be a self-conformal set that satisfies the open set condition in $\R^d$.  It is well known that the restricted Hausdorff measure $\mu:=\cH^{s}|_{K}$ is $s$-Ahlfors regular, where $s:=\dimH K$ and $\cH^{s}$ is the $s$-dimensional Hausdorff measure. Moreover, if in addition that $K$ is irreducible, then the Theorem 1.6 of \cite{MR4283274} illustrates that the measure $\mu$ is also absolutely decaying. Therefore, under these assumptions, the equality \eqref{dimequhypwin} is valid with $\supp(\mu)$ replaced by such a self-conformal set. In the Theorem \ref{conformalhyper} below, we will further remove the open set condition.
 \end{remark}
 
 However, Corollary \ref{corscplodimak} is suboptimal due to the fact that the lower Assouad dimension is not monotonic; that is to say that there might be a subset $K'\subseteq K$ such that $\lodimA K'>\lodimA K$. See Section 3.1 of \cite{MR4411274} for details. In light of this, the following result is more applicable in practice.

\begin{corollary}\label{dimhscapkgeqkia}
Let $L\subseteq\R^d$ be a closed hyperplane diffuse set and let $S\subseteq \R^d$ be hyperplane absolute winning on $L$. Then, for any  subset $K\subseteq L$, we have
\begin{equation*}
    \dimH S\cap K\geq\sup\big\{\lodimA K':\,\text{$K'\subseteq K$ is closed hyperplane diffuse}\big\}.
\end{equation*}
Here,  we adopt the convention that $\sup\emptyset:=0$.
\end{corollary}
\begin{proof}
    Let $L$, $S$ and $K$ be as give in the proposition. By combining (i)-(ii) of Proposition \ref{proofhypdi} and Proposition \ref{dimhlowdima}, for any closed hyperplane diffuse subset $K'\subseteq K$, we have $\dimH S\cap K\geq\dimH S\cap K'\geq\lodimA K'$. The proof is complete by taking the supremum of $\dimH K'$ over all closed hyperplane diffuse subsets of $K$. 
\end{proof}

This result provides the following two steps for proving $\dimH S\cap K\geq s$:
 \begin{itemize}
     \item[(i)] show that $S$ is hyperplane absolute winning on a closed and hyperplane diffuse set $L\supseteq K$; 
     \medskip

     \item[(ii)] for any $\epsilon>0$, find a closed and hyperplane diffuse subset $K'\subseteq K$ such that $\lodimA K'\geq s-\epsilon$.
 \end{itemize}
 The approach has been used in \cite{MR3904182} estimating the lower bound of $\dimH\BAD_d\cap K$ for a class of self-affine sets $K\subseteq\R^d$.  Using the same strategy, we obtain our first main result stated below, which establishes the full Hausdorff dimension of hyperplane absolute winning sets on irreducible self-conformal sets. A notable improvement here is the removal of the open set condition, thereby significantly extending the applicability of previous results (such as those mentioned in Remark~\ref{existresosc}) and providing an affirmative answer to Question~\ref{removeosc} (see Remark~\ref{answerremosc}). The proof is given in Section~\ref{proofselfcon}.

\begin{theorem}\label{conformalhyper}
Let $L\subseteq\R^d$ be a closed hyperplane diffuse set and let $S\subseteq \R^d$ be hyperplane absolute winning on $L$.  Then, for any irreducible self-conformal set $K\subseteq L$, we have $$\dimH S\cap K=\dimH K.$$
\end{theorem}
\begin{remark}
  { To establish the desired equality,  it suffices to show that $\dimH S\cap K\geq\dimH K$. As outlined above, the core of the proof lies in finding hyperplane diffuse subsets $K'\subseteq K$ such that the dimensions of these subsets are very close to that of $K$. A suitable path is provided by Falconer's classical construction \cite[Proof of Theorem 4]{MR969315}: for any $\epsilon>0$, he built a self-conformal subset $K'\subseteq K$ that satisfies the strong separation condition whose dimension is greater than $\dimH K-\epsilon$. However,  his argument does not clarify whether the subset is hyperplane diffuse provided that $K$ is irreducible. The gap is filled by our Proposition \ref{subsetirr}, which ensures the existence of hyperplane diffuse self-conformal subsets with the desired properties. }
\end{remark}
\begin{remark}\label{answerremosc}
    By Theorem 2.5 of \cite{MR2981929}, $\BAD_d$ is hyperplane absolute winning on $[0,1]^d$. As a consequence, $\dimH\BAD_d\cap K=\dimH K$ for any irreducible self-conformal set $K\subseteq[0,1]^d$. This answers the Question~\ref{removeosc} positively.
\end{remark}

This concludes our discussion demonstrating the power of hyperplane absolute game in determining the dimension of intersections.  With Theorem \ref{conformalhyper} in mind,    in order to address the Question \ref{quesofrpietwi}, we will examine the hyperplane absolute winning properties of twisted non-recurrent sets in the next part.

\subsection{Twisted non-recurrent sets are hyperplane absolute winning}\label{twishyperplane}  In this part, we will establish the hyperplane absolute winning property of the set $\tN(T,\cG)$ for a certain class of maps $T:[0,1]^d\to[0,1]^d$ and sequences $\cG=\{g_n:[0,1]^d\to[0,1]^d\}_{n\in\N}$. In order to state the result, we first introduce the notion of piecewise expanding maps.
\begin{definition}\label{defofpiee}
   Given $\kappa>0$, we say that the map $T:[0,1]\to[0,1]$ is \emph{piecewise $C^{1+\kappa}$  expanding} if it satisfies the following statements:
    \begin{itemize}
        \item[(i)] $T$ is \emph{piecewise $C^1$ monotonic}; that is to say that there exist a finite or infinite alphabet $\Sigma\subset\N$ and a collection of disjoint open sub-intervals $\{I_i\}_{i\in\Sigma}$ in $[0,1]$ satisfying that $$\overline{\bigcup_{i\in\Sigma}{I}_i}=[0,1]$$ and that $T$ is monotonic and  $C^1$ on  $I_i$ for any $i\in\Sigma$. 
        \vspace{1ex}

%         \item[(ii)] There exists $C>0$ and $\kappa>0$ such that
%         \begin{eqnarray*}
%         |T'(x)-T'(y)|\leq C|x-y|^{\kappa}
%         \end{eqnarray*}
% for any $i\in\Sigma$ and any $x,y\in I_i$.
%         \vspace{1ex}
        
        \item[(ii)] $|T'(x)|\geq 1$ for all $x\in\bigcup_{i\in\Sigma}I_i$ and there exist $\eta>1$ and $N_0\in\N$ such that
        \begin{eqnarray*}
            |(T^{N_0})'(x)|>\eta,\quad\forall\ x\in\bigcup_{\bfu\in\Sigma^{N_0}}I_{\bfu}\,,
        \end{eqnarray*}
        where $\Sigma^n$ is the set of all words with length $n$ over $\Sigma$ for any $n\in\N$, and $I_{\bfu}$ denotes  the \emph{cylinder set associated with $\bfu=u_1u_2\cdot\cdot\cdot u_n\in\Sigma^n$ $(n\in\N)$ } defined as $$I_{\bfu}:=I_{u_1}\cap T^{-1}(I_{u_2})\cap\cdot\cdot\cdot\cap T^{-(n-1)}(I_{u_n})\,.\vspace{1ex}$$
        \item[(iii)] there exist $C>0$  such that
        \begin{eqnarray*}
        |T'(x)-T'(y)|\leq C|x-y|^{\kappa}
        \end{eqnarray*}
for any $i\in\Sigma$ and any $x,y\in I_i$.
    \end{itemize}
Moreover, we say that \emph{$T$ has finite branches} if $\#\Sigma<+\infty$ whereas \emph{infinite branches} if $\#\Sigma=+\infty$.
\end{definition}

The following theorem  represents our main one-dimensional result on winning property of the set $\tN(T,\cG)$. 

\begin{theorem}\label{abs1dim}
    Let $T:[0,1]\to[0,1]$ be a piecewise $C^{1+\kappa}$  expanding map that satisfies one of the following statements:
  \begin{itemize}
            \item[(A)] %(finite branches) 
            $T$ has finite branches.
            \medskip

            \item[(B)] %(full branches) 
             $T(I_i)=(0,1)$ for any $i\in\Sigma$.
        \end{itemize}
 Let $\cG:=\{g_n:[0,1]\to[0,1]\}_{n\in\N}$ be a sequence of functions satisfying that 
    \begin{eqnarray}\label{unilip}
        |g_n(x)-g_n(y)|\leq C|x-y|,\quad\forall\ n\in\N,\ \forall\ x,y\in[0,1]
        \end{eqnarray}
        for some constant $C>0$.
    Then the set $\tN(T,\cG)$ is hyperplane absolute winning on $[0,1]$.
\end{theorem}
{\begin{remark}\label{onedimtherek} Several comments are in order.
\begin{itemize}
\item[(i)] The theorem provides a non-trivial generalization of a recent result by Hu, Li and Yu \cite{hu2025} extending their framework from badly approximable sets to a broader setting of twisted non-recurrent sets. Their proof  is based on a key observation: given a fixed point $y\in[0,1]$, they found that
  \begin{equation*}
      B\cap \left(\bigcup_{i=n}^{n+N}\big\{x\in[0,1]:|T^i(x)-y|<r\big\}\right)
  \end{equation*}
  is either empty or an interval provided that the ball $B\subseteq[0,1]$, the radius $r>0$ and $n,N\in\N$  satisfy some relations. However, their argument relies crucially on the quantity $$\min\{|T^{i}(y)-y|:1\leq i\leq N\}\setminus\{0\},$$ which does not easily  transfer to the twisted setting where the target point varies as $y=g_i(x)$. To bridge the gap, our Proposition \ref{keystr}, inspired by \cite[Lemma 3.2]{MR2644378}, avoids analyzing the fine  structure of
  \begin{equation*}
      B\cap \left(\bigcup_{i=n}^{n+N}\big\{x\in[0,1]:|T^i(x)-g_n(x)|<r\big\}\right),
  \end{equation*}
  thereby enabling our general result.\medskip
 
    \item[(ii)] In the Section 5.4 of the paper \cite{MR2644378} by  Broderick, Bugeaud, Fishman, Kleinbock and Weiss,  they informally raised the following question:  \emph{whether $\BAD(T,\mathcal{Y})$ is winning on $K$  whenever $T=T_{\beta}$   is the $\beta$-transformation $T_{\beta}(x):=\beta x\ \tmod \ 1$ defined on $[0,1]$ with $\beta>1$, $\mathcal{Y}=\{y_n\}_{n\in\N}$ is a sequence of points in $[0,1]$ and $K\subseteq[0,1]$ is the support of an absolutely decaying measure?} Recall that if $\mu$ is an absolutely decaying measure on $\R^d$, then $\supp(\mu)$ is hyperplane diffuse. Also note that for any $\beta>1$, the map $T_{\beta}$ is a piecewise $C^{1+\kappa}$ expanding map that satisfies the assumption (A) in Theorem \ref{abs1dim}. Therefore, on combining with Corollary  \ref{corscplodimak}, our Theorem \ref{abs1dim} implies that $\BAD(T_{\beta},\mathcal{Y})$ is winning on $K=\supp(\mu)$ for any $\beta>1$ and absolutely decaying measure $\mu$ supported on $[0,1]$. This answers their question positively. \medskip
     \item[(iii)] In the final stage of writing the paper, we surprisingly found that a very recent work \cite{lambert2025} had independently shown a result that is highly related to our Theorem \ref{abs1dim}: they proved that $\tN(T,\cG)$ is hyperplane absolute winning on $[0,1]$ whenever $T$ is the $\beta$-transformation with $\beta>1$ or Gauss map on $[0,1]$ and $\cG=\{g_n\}_{n\in\N}$ is equicontinuous on $[0,1]$.  Clearly, their considered maps are special cases of piecewise $C^{1+\kappa}$ expanding maps that satisfy the assumption (A) or (B) in our Theorem \ref{abs1dim}, whereas the condition they posed on $\cG$ is slightly weaker than our \eqref{unilip}.
    \end{itemize}
\end{remark}
}
As we will demonstrate below, Theorem $\ref{abs1dim}$ also holds for a class of high-dimensional product systems. We first clarify the notion of product maps. Given $T_1:X_1\to X_1$ and $T_2:X_2\to X_2$,  we define the associated \emph{product map} $T_1\otimes T_2$ as 
\begin{equation}\label{product}
    T_1\otimes T_2(x_1,x_2)=(T_1(x_1),T_2(x_2)),\qquad\forall\ (x_1,x_2)\in X_1\times X_2.
\end{equation}
In order to state the high-dimensional result,  let us introduce the following setting.
\begin{itemize}
    \item  Let $d$ be a positive integer and let $X=[0,1]^d$  with the metric induced by the maximal norm $\|\cdot\|$ defined as in \eqref{maxnorm}.\medskip
\item Let $T:[0,1]^d\to[0,1]^d$ be a map of the form $T=T_1\otimes T_2$,
where $T_1:[0,1]\to[0,1]$ is a $C^{1+\kappa}$ piecewise expanding map on $[0,1]$ that satisfies   (A) or  (B) appearing in the statement of Theorem \ref{abs1dim}, and $T_2$ is a self map on $[0,1]^{d-1}$. Here, we adopt the convention that $[0,1]^{0}:=\emptyset$.
\medskip

\item Let $\cG:=\{g_n:[0,1]^d\to[0,1]^d\}_{n\in\N}$ satisfy that
\begin{eqnarray}\label{conslipgn}
     \hspace{5ex}\|g_{n}(\bfx)-g_{n}(\bfy)\|\leq C\|\bfx-\bfy\|,\qquad\forall\ \bfx,\bfy\in[0,1]^d,\ \forall\ n\in\N
\end{eqnarray}
for some constant $C>0$.
\end{itemize}
\medskip

\noindent It is easily seen that the above setting reduces to that of Theorem \ref{abs1dim} in the case $d=1$.  Therefore, the following result extends Theorem \ref{abs1dim} to high-dimensional space. 

\begin{theorem}\label{highdim}
    Let $T$ and $\cG$ be as above. Then the set $\tN(T,\cG)$ is hyperplane absolute winning on $[0,1]^d$.
    % Then the Hausdorff dimension of the set $N(T,\cG)$ equals $d$. Moreover, if in addition
    % \begin{eqnarray*}
    %     g_n(\bfx)=(g_{n,1}(x_1),g_{n,2}(\bfx_2))\qquad\forall\ \bfx=(x_1,\bfx_2)\in[0,1]\times[0,1]^{d-1},\ \forall\ n\in\N,
    % \end{eqnarray*}
    % where $\{g_{n,1}:[0,1]\to[0,1]\}_{n\in\N}$ is a sequence of functions with uniform Lipschitz constants and $\{g_{n,2}\}_{n\in\N}$ is a sequence of self maps on $[0,1]^{d-1}$, then the set $N(T,\cG)$ is  hyperplane absolute winning on $[0,1]^d$.
\end{theorem}

On combining this with Corollary \ref{dimhscapkgeqkia} and Theorem \ref{conformalhyper}, we obtain the following result on the sizes of intersections of $\tN(T,\cG)$ with fractals, which answers partially the Question~\ref{quesofrpietwi}  and a question asked by Li, Liao, Velani and Zorin \cite{MR4572386} for diagonal real matrix transformations (see Remark \ref{cor1.4remk} below).

\begin{corollary}\label{dimintwitself}
    Let $T$ and $\cG$ be as in Theorem \ref{highdim}. Then, for any subset $K\subseteq[0,1]^d$, we have
    \begin{equation*}
    \dimH \tN(T,\cG)\cap K\geq\sup\big\{\lodimA K':\,\text{$K'\subseteq K$ is closed hyperplane diffuse}\big\}.
\end{equation*}
In particular, for any irreducible self-conformal set $K\subseteq [0,1]^d$, we have $$\dimH \tN(T,\cG)\cap K=\dimH K.$$
\end{corollary}
{\begin{remark}\label{cor1.4remk}
Let $E\subseteq[0,1]^d$ and $\bfy\in[0,1]^d$. Denote
 \begin{equation*}
    \BAD(T,\bfy,E):=\left\{\bfx\in[0,1]^d\left|\begin{aligned}
         &\exists\,c(\bfx)>0\ \text{such that}\ T^n(\bfx)\notin\bfy+c(\bfx)E\ \tmod\ 1\\ &\text{for any sufficiently large $n\in\N$}
     \end{aligned}\right.\right\}.
 \end{equation*}
 In the Section 5.4 of the paper \cite{MR4572386}, Li, Liao, Velani and Zorin asked the next mentioned question which is in line with our Question~\ref{quesofrpietwi} (i): \emph{if $T:[0,1]^d\to[0,1]^d$ is a real matrix transformation as in \eqref{txeqaxmod1} whose modules of eigenvalues of the corresponding matrix are all greater than one, $\bfy$ is a point in $[0,1]^d$ and $E\subseteq[0,1]^d$ satisfies their bounded condition ($\mathbf{B}$), does it hold that $\dimH \BAD(T,\bfy, E)=d$?} In the case $E=B(\mathbf{0},2)$, we have $\BAD(T,\bfy,E)=\BAD(T,\bfy)$ for any $\bfy\in[0,1]^d$. Moreover, since $B(\mathbf{0},2)\supseteq[0,1]^d$, it follows that $\BAD(T,\bfy)\subseteq\BAD(T,\bfy,E)$ for any subset $E\subseteq[0,1]^d$. Therefore, to establish $\dimH\BAD(T,\bfy,E)=d$, it suffices to show that $\BAD(T,\bfy)$ has full Hausdorff dimension. Naturally, we are more plausible that the equality $\dimH\BAD(T,\bfy)=d$ holds  only assuming that the matrix associated with $T$ has at least one eigenvalue with modules greater than one.   In the following, consider the diagonal matrix transformation $$T(\bfx):=(\beta_1 x_1,\beta_2 x_2,...,\beta_d x_d)\ \tmod\ 1\qquad\forall \ \bfx\in[0,1]^d,$$ where $\beta_i\in\R$ for any $1\leq i\leq d$ and $|\beta_{i_0}|>1$ for at lease one subscript $1\leq i_0\leq d$. Assuming without loss of generality that $|\beta_1|>1$,   the first coordinate map $T_1x_1:=\beta_1x_1\ \tmod \ 1$ becomes a $C^{1+\kappa}$ piecewise expanding map on $[0,1]$ that satisfies the assumption (A)  of Theorem \ref{abs1dim}. Consequently,  Theorem \ref{highdim} implies that $\BAD(T,\bfy)$ is hyperplane winning for any $\bfy\in[0,1]^d$ and thus $\dimH\BAD(T,\bfy)=d$. This partially answers the question in \cite{MR4572386} for diagonal matrix transformations. We in fact establish a more general result: for any $\cG=\{g_n\}_{n\in\N}$ that satisfies \eqref{conslipgn} and a wide class of subsets  $K\subseteq [0,1]^d$ including irreducible self-conformal sets, we have $\dimH\tN(T,\cG)\cap K=\dimH K$. This also provides a partial response to Question~\ref{quesofrpietwi}. Nevertheless, the techniques developed in the present work do not readily extend to broader setting of general real matrix transformations. Therefore, the question in~\cite{MR4572386}, as well as our Question~\ref{quesofrpietwi}, still remains largely open.
\end{remark}
}

\subsection{Generalizations: non-autonomous settings} In this subsection,  we extend Theorem \ref{highdim} to non-autonomous settings.   A \emph{non-autonomous dynamical system} is a pair $(X,\cT)$, where $X$ is a metric space and  $\cT=\{T_n\}_{n\in\N}$ is a sequence of self-maps on $X$.
%Let $(X,d)$ be a metric space and let $\cF=\{f_n\}_{n\in\N}$ be a sequence of self-maps on $X$. A pair of $(X,\cF)$ is called a non-autonomous dynamical system.
For any $i,N\in\N$ and $x\in X$, define
\begin{equation*}
 \sfT_{i,N}(x):=T_{i+N-1}\circ T_{i+N-2}\circ\cdot\cdot\cdot\circ T_i\,(x),\quad\sfT_{i,0}(x):=\id_{X}(x)=x.
\end{equation*}
 Let $\cG=\{g_n:X\to X\}_{n\in\N}$  and let $\bfd$ be the metric on $X$. We define the \emph{twisted non-recurrent set $\tN(\cT,\cG)$ associated with $\cT$ and $\cG$} as
\begin{equation*}
   \mathrm{N}(\cT,\cG)  :=\left\{x\in X:\liminf_{n\to\infty}\bfd\big(\sfT_{1,n}(x),g_n(x)\big)>0\right\}.
\end{equation*}

% In order to state the analogous result of Theorem \ref{highdim} for non-autonomous systems, we first introduce various useful notations. To begin with, let $\cT=\{T_n\}_{n\in\N}$ be a sequence of piecewise $C^1$ monotonic maps on $[0,1]$ as stated in Definition \ref{defofpiee} (i). For any $n\in\N$, let $\Sigma_n\subseteq\N$ and $\{I_{i}(n)\}_{i\in\Sigma_n}$ be  the  alphabet and disjoint open intervals   associated with $T_n$, respectively, as described in Definition \ref{defofpiee} (i). Next, for any $i,N\in\N$, denote
% \begin{equation}\label{SigmaiN}
% \Sigma(i,N):=\Sigma_{i}\times\Sigma_{i+1}\times\cdot\cdot\cdot\Sigma_{i+N-1}.
% \end{equation}
%  Given a word $\bfu=(u_i,u_{i+1},...,u_{i+N-1})\in\Sigma(i,N)$, define the associated cylinder set $I_{\bfu}(i,N)$ as the following
%  \begin{equation}\label{IuiN}
%      I_{\bfu}(i,N):=\bigcap_{n=0}^{N-1}\sfT_{i,n}^{-1}\big(I_{u_{i+n}}(i+n)\big).
%  \end{equation}
Let $\kappa>0$, we say that  $\cT=\{T_n\}_{n\in\N}$ is a sequence of \emph{uniformly piecewise $C^{1+\kappa}$  expanding maps} on the unit interval $[0,1]$ if
\begin{itemize}
        \item[(i)] $T_n$ is piecewise $C^1$ monotonic for any $n\in\N$; \vspace{1ex}
        \item[(ii)] there exists $C>0$ such that
         \begin{equation}\label{c1kappalln}
        |T_n'(x)-T_n'(y)|\leq C|x-y|^{\kappa},\quad\forall\ n\in\N,\ \forall\ i\in\Sigma_n,\ \forall\ x,y\in I_i(n),\vspace{1ex}
    \end{equation}
    where $\Sigma_n\subseteq\N$ and $\{I_{i}(n)\}_{i\in\Sigma_n}$ are  the  alphabet and disjoint open intervals   associated with $T_n$, respectively, as described in Definition \ref{defofpiee} (i);
    \medskip
    \item[(iii)] $|T_n'(x)|\geq1$ for all $n\in\N$ and $x\in\bigcup_{i\in\Sigma_n}I_i(n)$. Moreover there exist $N_0\in\N$ and $\eta>1$ such that 
     \begin{equation}\label{unifoexpad}
     \big|\big(\sfT_{i,N_0}\big)'(x)\big|>\eta,\qquad \forall\ i\in\N,\ \forall \ \bfu\in\Sigma(i,N_0),\ \forall\ x\in I_{\bfu}(i,N_0),\vspace{1ex}
 \end{equation}
 where
 \begin{eqnarray}\label{SigmaiN}
     \Sigma(i,N):=\Sigma_{i}\times\Sigma_{i+1}\times\cdot\cdot\cdot\Sigma_{i+N-1}
 \end{eqnarray} and
  \begin{equation}\label{IuiN}
     I_{\bfu}(i,N):=\bigcap_{n=0}^{N-1}\sfT_{i,n}^{-1}\big(I_{u_{i+n}}(i+n)\big) 
 \end{equation}
 for any $\bfu=u_iu_{i+1}\cdots u_{i+N-1}\in\Sigma(i,N)$.
    \end{itemize}
%where $\sfT_{\cF}(i,n)^{-1}(A):=\{x\in[0,1]:\sfT_{\cF}(i,j,x)\in A\}$.

Now, we propose the following  framework that is analogous to that of Theorem~\ref{highdim} for non-autonomous systems. 

\begin{itemize}
\item Let $X=[0,1]^d$ with the metric induced by the maximal norm.
\medskip

    \item Let $\cT:=\{T_n=T_{n,1}\otimes T_{n,2}\}_{n\in\N}$ be a sequence of self-maps on $[0,1]^d$, where $T_{n,1}:[0,1]\to[0,1]$ and $T_{n,2}:[0,1]^{d-1}\to[0,1]^{d-1}$ for any $n\in\N$.\medskip %Here, we use  $\cF_1\otimes\cF_2$ to  denote the sequence of the functions $\{f_{n,1}\otimes f_{n,2}\}_{n\in\N}$, where $f_1\otimes f_2$ is the product map induced by $f_1$ and $f_2$ (see \eqref{product}).

    \item Suppose further that $\cT_1:=\{T_{n,1}\}_{n\in\N}$  is a sequence of uniformly  piecewise $C^{1+\kappa}$ expanding maps on $[0,1]$ that satisfy one of the following conditions:\medskip

    \begin{itemize}
        \item[($\mathrm{A}'$)] $\#\{T_{n,1}:n\in\N\}<+\infty$ and $T_{n,1}$ has finite branches  for any $n\in\N$.
        % all the maps in $\cT_1$ satisfy assumption (A) appearing in the statement of Theorem \ref{abs1dim}; that is to say that 
        \medskip

        \item[($\mathrm{B}'$)]  
        % all the maps in $\cT_1$ satisfy assumption (B)  in  Theorem \ref{abs1dim}; that is to say that
 $T_{n,1}(I_i(n))=(0,1)$ for all $n\in\N$ and $i\in\Sigma_n$, where $\Sigma_n\subseteq\N$ and $\{I_{i}(n)\}_{i\in\Sigma_n}$ are  the  alphabet and disjoint open intervals   associated with $T_{n,1}$, respectively. % as stated in Definition \ref{defofpiee} (i).
    \end{itemize}
    \medskip
    
    \item Let $\cG:=\{g_n:[0,1]^d\to[0,1]^d\}_{n\in\N}$ satisfy \eqref{conslipgn}.
\end{itemize}

Note that in the case $T_{n}=T$  for all $n\in\N$, the above framework coincides with that of Theorem \ref{highdim}. In light of this,  the following result  generalizes Theorem \ref{highdim} to  non-autonomous dynamical systems on $[0,1]^d$. We refer to Section~\ref{pfofthemainnon} for the proof.

\begin{theorem}\label{onedimnonaut}
    Let $\cT$ and $\cG$ be as above. Then the set $\tN(\cT,\cG)$ is hyperplane absolute winning on $[0,1]^d$.
\end{theorem}

In the end, we present a simple application to normality in the $Q$-Cantor setting.

\begin{example}
    Let $Q=\{q_n\}_{n\in\N}\subseteq\N\cap[2,+\infty)$ and let $T_n(x):=q_nx\ \tmod\ 1$ for any $x\in[0,1]$ and $n\in\N$. This defines a non-autonomous system $([0,1],\cT)$ called $Q$-Cantor system, where $\cT:=\{T_n\}_{n\in\N}$. A real number $x\in[0,1]$ is said to be \emph{$Q$-distribution normal} if the sequence $\{\sfT_{1,n}(x)\}_{n\in\N}$ is uniformly distributed in $[0,1]$; that is
    \begin{equation*}
        \lim_{N\to\infty}\frac{\#\{1\leq n\leq N:\sfT_{1,n}(x)\in(a,b)\}}{N}=b-a,\qquad\forall\ 0\leq a<b\leq 1.
    \end{equation*}
    Let $\sfN$ be the set of $x\in[0,1]$ that is \emph{not} $Q$-distribution normal. As noted in the Remark 1.11 of \cite{MR2999070}, the Lebesgue measure of $\sfN$ is zero. Next, we are concerned about the winning property of $\sfN$. Observe that it is related to the twisted non-recurrent set via
    \begin{equation*}
        \big\{x\in[0,1]:\liminf_{n\to\infty}|\sfT_{1,n}(x)-y|>0\big\}\subseteq\sfN,\quad\forall\ y\in[0,1].
    \end{equation*}
    Then, the Theorem 1.1 of \cite{MR2644378} regarding lacunary sequences implies that $\sfN$ is winning on the support of an absolutely decaying measure in $[0,1]$. In terms of intersections with self-conformal sets, it implies that $\dimH \sfN\cap K=\dimH K$ for any self-conformal set $K\subseteq[0,1]$ that satisfies the open set condition. Since $q_n\in\N$ for any $n\in\N$, the sequence $\{T_n\}_{n\in\N}$ thereby satisfies the assumption $(\mathrm{B}')$.  Then, our Theorem \ref{onedimnonaut} together with Theorem \ref{conformalhyper}  allows us to remove the open set condition, therefore strengthening the previous conclusion. For further research on the normality in the $Q$-Cantor setting,  we refer to \cite{MR4127855,MR3345187,MR3705133,MR2827169,MR3325427,MR3274409,MR2999070} and references therein.
\end{example}

\subsection{Organizations}
The remaining parts of the paper are organized as follows. In Section \ref{defofselfcon}, we introduce the basics of self-conformal sets. The proof of Theorem \ref{conformalhyper} is then given in Section \ref{proofselfcon}. The Section \ref{UPCEM} is devoted to basic properties of uniformly piecewise $C^{1+\kappa}$ expanding maps. In Section \ref{pfofthemainnon}, we present the proof of Theorem \ref{onedimnonaut}.

\section{self-conformal sets and basic facts}\label{defofselfcon}
In this section, we review the basics of self-conformal sets in preparation for the proof of Theorem \ref{conformalhyper}. Throughout, let $|\cdot|$ denote the Euclidean norm (namely $L^2$ norm) on $\R^d$.

\begin{definition}\label{defofconset}
   Let $m\in\N\cap[2,+\infty]$ and $\kappa>0$. We say that $\Phi=\{\varphi_i\}_{i=1}^m$ is a \emph{$C^{1+\kappa}$ self-conformal IFS} on $\R^d$ if there exists a connected bounded open set $\Omega\subseteq\R^d$ such that
   \begin{itemize}
       \item[(i)]  $\varphi_i:\Omega\to\R^d$ is $C^1$ for any $1\leq i\leq m$.
       \medskip

       \item[(ii)] $\varphi_i$ is conformal on $\Omega$ for any $1\leq i\leq m$; that is $\varphi_i'(\bfx)\neq0$ for any $\bfx\in\Omega$ and
       \begin{equation*}
|\varphi_i'(\bfx)\bfy|=\|\varphi_i'(\bfx)\|\cdot|\bfy|,\qquad\forall\ \bfx\in\Omega,\ \forall \ \bfy\in\R^d,\ \forall \ 1\leq i\leq m.
       \end{equation*}
       Here, $f'(\bfx)$ is the Jacobian matrix and $\|A\|:=\sup_{|\bfx|\leq 1}|A\bfx|$ for a $d\times d$ matrix $A$.\medskip

       \item[(iii)] there exists $C>0$ such that
       \begin{equation*}
           \big|\|\varphi_i'(\bfx)\|-\|\varphi_i'(\bfy)\|\big|\leq C|\bfx-\bfy|^{\kappa},\qquad\forall \ \bfx,\bfy\in\Omega,\ \forall \ 1\leq i\leq m.\medskip
       \end{equation*}

       \item[(iv)] for any $1\leq i\leq m$, $\varphi_i$ is contractive and
       \begin{equation*}
           0<\inf_{\bfx\in\Omega}\|\varphi'(\bfx)\|\leq \sup_{\bfy\in\Omega}\|\varphi'(\bfx)\|<1.
       \end{equation*}
   \end{itemize}
   It is well-known (Hutchinson \cite{hutchinson1981}) that there exists unique compact set $K\subseteq\Omega$, called \emph{self-conformal set}, such that
   \begin{equation*}
       K=\bigcup_{i=1}^m\varphi_i(K).
   \end{equation*}
\end{definition}
\begin{remark}\label{conforremark}
    Let $\Omega\subseteq\R^d$ be a connected open set. We review the rigidity of conformal maps on $\Omega$. In the case $d=1$, a map $f:\Omega\to\R$ is conformal if and only if $f\in C^1(\Omega)$ and $f'(x)\neq0$ for all $x\in\Omega$. In the two dimensional case, by viewing $\R^2$ as the complex plane $\bC$, an injective map $f:\Omega\to \bC$ is conformal if and only if $f$ is holomorphic (or anti-holomorphic) on $\Omega$. 
When $d\geq3$, by Liouville's theorem (see \cite[Section 3.8]{gehring2017}), a map $f:\Omega\to\R^d$ is conformal if and only if it is a restriction to $\Omega$ of a  Möbius transformation on $\overline{\R}^d:=\R^d\cup\{\infty\}$, that is
\[
    f(\bfx)=\mathbf{b}+\frac{c}{|\bfx-\mathbf{a}|^{\epsilon}}\cdot A(\bfx-\mathbf{a}),
\]
where $\mathbf{a},\mathbf{b}\in\R^d$, $c\in\R$, $\epsilon\in\{0,2\}$ and $A$ is a $d\times d$ orthogonal matrix. In the proof of Theorem \ref{conformalhyper}, we only use the fact that a conformal map $f:\Omega\to\R^d$ is real-analytic in the case $d\geq2$.
\end{remark}

The following separation conditions are useful in the study of self-conformal sets.

\begin{definition}\label{defofosc}
    Let $\Phi=\{\varphi_i\}_{i=1}^m$ be a $C^{1+\kappa}$ self-conformal IFS on $\R^d$ and $K\subseteq\R^d$ be the associated self-conformal set. We say that $\Phi$ satisfies \emph{strong separation condition} (SSC) if $\varphi_i(K)\cap\varphi_j(K)=\emptyset$ for any $1\leq i\neq j\leq m$, and \emph{open set condition} (OSC) if there exists an open set $O\subseteq\Omega$ such that 
    \begin{equation*}
        \varphi_i(O)\subseteq O,\quad\varphi_i(O)\cap\varphi_j(O)=\emptyset
    \end{equation*}
    for any $1\leq i\neq j\leq m$.
\end{definition}

\noindent It is clear that  SSC implies OSC. Under the open set condition, it is well-known that the restricted Hausdorff measure $\cH^{s}|_K$ is $s$-Ahlfors regular, where $s:=\dimH K$. Recall that a self-conformal set in $\R^d$ is irreducible if it is not contained in a $(d-1)$-dimensional real-analytic submanifold of $\R^d$. As noted in Remark~\ref{existresosc}, an irreducible self-conformal set $K\subseteq\R^d$ that satisfies the open set condition is the support of an absolutely decaying measure. This together with the Proposition 5.1 of \cite{MR2981929} implies that such $K$ is hyperplane diffuse.
% \begin{definition}
%     We say that a self-conformal set $K\subseteq\R^d$ is \emph{irreducible} if it is not contained in any $(d-1)$-dimensional real-analytic submanifold of $\R^d$.
% \end{definition}
% \begin{remark}\label{irrerek}
% Let $K\subseteq\R^d$ be an irreducible self-conformal set that satisfies the open set condition.  As noted in Remark~\ref{existresosc}, $K$ is the support of an absolutely decaying measure. This together with the Proposition 5.1 of \cite{MR2981929} implies that $K$ is hyperplane diffuse.
% Then  $K$ is irreducible if and only if
% \begin{itemize}
%     \item[(i)]  $K$ is not a single point when $d=1$.
%     \medskip

%     \item[(ii)] $K$ is not contained in a finite disjoint union of analytic curves when $d=2$; see Proposition 3.1 (iii) of \cite{huang2025}. Here, we say that a subset $\Gamma\subseteq\R^2$ is an analytic curve if $\Gamma=f([0,1]\times\{0\})$ for some conformal map $f:\mathcal{O}\to\R^2$ where $\mathcal{O}\subseteq\R^2$ is an open set containing $[0,1]\times\{0\}$.
%     \medskip

%     \item[(iii)] $K$ is not contained in a $(d-1)$-dimensional affine hyperplane or $(d-1)$-dimensional sphere when $d\geq3$; see Proposition 3.1 (i)-(ii) of \cite{huang2025}.
% \end{itemize}
% \end{remark}

We end up the section with the following basic properties of self-conformal sets. Let $\Sigma=\{1,2,...,m\}$. For any $n\in\N$, denote $\Sigma^n$ the set of words of length $n$ over $\Sigma$. The notations $\Sigma^*$ and $\Sigma^{\N}$ stand for, respectively, the collection of all finite words and infinite words over $\Sigma$. Given a finite word $I=i_1i_2\cdots i_n\in\Sigma^{*}$, write $\varphi_I:=\varphi_{i_1}\circ\varphi_{i_2}\circ\cdots\varphi_{i_n}$ and $K_I:=\varphi_I(K)$.

\begin{lemma}\label{basiclem}
    Let $K\subseteq\R^d$ be a self-conformal set generated by a $C^{1+\kappa}$ self-conformal IFS $\Phi=\{\varphi_i\}_{i=1}^m$ on $\R^d$. Then
    \begin{itemize}
        \item[(i)] there exists a constant $C>1$ such that
        \begin{eqnarray}\label{boundeddis}
            C^{-1}\|\varphi_I'\|\leq\|\varphi_I'(\bfx)\|\leq C\|\varphi_I'\|,\qquad I\in\Sigma^*,\ \forall \ \bfx\in \Omega,
        \end{eqnarray}
        where $\|f'\|:=\sup_{\bfx\in\Omega}\|f'(\bfx)\|$.
        \medskip

        \item[(ii)] there exist a bounded open set $U\subseteq\Omega$ and a constant $C>1$ such that
        \begin{eqnarray*}
            \overline{U}\subseteq\Omega,\quad\varphi_i(U)\subseteq U,\quad\forall \ i=1,2,...,m
        \end{eqnarray*}
        and
        \begin{equation}\label{uniformbilip}
            C^{-1}\|\varphi_I'\|\cdot|\bfx-\bfy|\leq|\varphi_I(\bfx)-\varphi_I(\bfy)|\leq C\|\varphi_I'\|\cdot|\bfx-\bfy|
        \end{equation}
        for any $I\in\Sigma^*$ and $\bfx$, $\bfy\in U$.\medskip
 
        \item[(iii)] there exists a constant $C>1$ such that
        \begin{equation}\label{diamofki}
            C^{-1}\|\varphi_I'\|\leq |K_I|\leq C\|\varphi_I'\|,\qquad\forall \ I\in\Sigma^*.
        \end{equation}
    \end{itemize}
\end{lemma}
\begin{proof}
    The statements (i)-(ii) follows respectively from Lemma 2.1 and 2.2 of \cite{MR1479016}. The statement (iii) is a direct consequence of (ii).
\end{proof}

\section{Proof of Theorem \ref{conformalhyper}}\label{proofselfcon}

In view of the Corollary~\ref{dimhscapkgeqkia}, the Theorem \ref{conformalhyper} is a direct consequence of the following result. 

\begin{proposition}\label{subsetirr}
    Let $d$ be a positive integer and let $K\subseteq\R^d$ be an irreducible self-conformal set. Let $s:=\dimH K$. Then for any $0<\epsilon<s$, there exists $K'\subseteq K$ that satisfies the following statements:
    \begin{itemize}
        \item[(i)] $K'$ is a hyperplane diffuse set generated by a $C^{1+\kappa}$ conformal IFS that satisfies the strong separation condition.
        \medskip

        \item[(ii)] the lower Assouad dimension of $K'$ is greater than $s-\epsilon$.
    \end{itemize}
\end{proposition}

Before the proof, we introduce various  notions that will be used. 
     For any $k=1,...,d-1$, let $G(d,k)$ be the collection of all $k$-dimensional subspaces in $\R^d$ and let $A(d,k)$ be the collection of all $k$-dimensional affine hyperplanes in $\R^d$. They are related via     \begin{equation*}
         A(d,k)=\left\{V+y:V\in G(d,k),\,y\in V^{\perp}\right\},
   \end{equation*}
   where $V^{\perp}$ denotes the  orthogonal complement of $V\in G(d,k)$.
    Note that for any $V\in G(d,k)$ there exists a unique linear operator $P_V:\R^d\to\R^d$ such that $P_V|_V:V\to V$ is an identity map on $V$ and $P_V(V^{\perp})=\{\mathbf{0}\}$.   The distance between $V_1+y_1$ and $V_2+y_2$, where $V_1,V_2\in G(d,k)$ and $y_1,y_2\in V^{\perp}$, is defined as
     \begin{equation*}
         \bfd(V_1+y_1,V_2+y_2):=\|P_{V_1}-P_{V_2}\|+|y_1-y_2|.
     \end{equation*}
    This makes $A(d,k)$ a locally compact metric space. Next, we review the basic notion of the tangent space coming from differential geometry. Let $M\subseteq \R^d$ be a $\ell$-dimensional smooth submanifold and $\bfx\in M$. Let $\varphi:U\to \varphi(U)$ be a diffeomorphism and $\psi:\varphi(U)\to U$ be its inverse, where $U\subseteq M$ is an open set relative to $M$ containing $\bfx$ and $\varphi(U)$ is an open set in $\R^{\ell}$. The \emph{tangent space} of $M$ at $\bfx$ is defined as
    \begin{equation*}
        T_{\bfx} M:=\left\{\psi'(\varphi(\bfx))\bfy:\bfy\in\R^{\ell}\right\}.
    \end{equation*}
    It is clear that $T_{\bfx}M$ is a linear subspace of $\R^d$ independent of the choice of $\varphi$.

\begin{proof}[Proof of Proposition \ref{subsetirr}]
     For any $r>0$, let 
     \begin{equation*}
         \Lambda_r:=\big\{I\in\Sigma^*:|K_I|<r\leq |K_{I^-}|\big\}
     \end{equation*}
     and let $\cI_r$ to be a subset of $\Lambda_r$ that satisfies the following statements:
     \begin{itemize}
         \item $K_I\cap K_J=\emptyset$ for any $I\neq J\in\cI_r$.
         \medskip
         \item if $\cI\supseteq\cI_r$ is another subset of $\Lambda_r$ that satisfies the above property with $\cI_r$ replaced by $\cI$, then $\cI=\cI_r$.
     \end{itemize}
     Note that the choice of such $\cI_r$ may not be unique. Let $\Phi_r:=\{\varphi_I\}_{I\in\cI_r}$ and let $K_r$ be its associated self-conformal set.  By the definition of $\cI_r$, the set $K_r$ is $2r$-dense in $K$. This means  that for any $\bfx\in K$, there exists $\bfy\in K_r$ such that $|\bfx-\bfy|\leq 2r$. We next verify  the statements (i)-(ii) of the lemma for $K_r$.\medskip

     \noindent $\bullet$ \emph{Proving that (i) is true for $K_r$ when $r$ is sufficiently small.} By the construction of $\cI_r$, it is clear that $\Phi_r$ satisfies the strong separation condition. The proof of the irreducible part is divided into two cases.
     \medskip

     \emph{Case 1:  $d=1$}. Recall that $K$ is irreducible, then there exists $r_1>0$ such that $\#\cI_r\geq2$ for any $0<r<r_1$.  Since $\Phi_r$ satisfies the strong separation condition, it  follows that $K_r$ is not a single point. In the one-dimensional case, this is equivalent to the  irreducibility of $K_r$.
     \medskip

     \emph{Case 2: $d\geq2$}. Suppose in contradiction that there exists a sequence of positive numbers $\{r_n\}_{n\in\N}$ such that $r_n\to 0$ as $n\to\infty$ and $K_{r_n}$ is not irreducible for any $n\in\N$. Then, for any $n\in\N$, $K_{r_n}\subseteq M_n$ where $M_n$ is a $(d-1)$-dimensional real-analytic submanifold in $\R^d$.  Fix $\bfx_n\in K_{r_n}$ and let $J_n=I_{1,n}I_{2,n}\cdots$ be the infinite word in $\cI_{r_n}^{\N}$ such that $\pi_n(J_n)=\bfx_n$, where $\pi_n:\cI_{r_n}^{\N}\to K_{r_n}$ denotes the coding map associated with $\Phi_{r_n}$. For any $n\in\N$, observe that
     \begin{itemize}
         \item[(a)] by Taylor's formula, for any $\epsilon>0$, there exists $\delta=\delta(\epsilon,n)>0$ such that 
         \begin{equation*}
             B(\bfy_n,\delta')\cap M_n\subseteq(\bfy_n+T_{\bfy_n}M_n)_{\epsilon\delta'},\qquad\forall \ 0<\delta'\leq \delta,
         \end{equation*}
          where $\bfy_n:=\varphi_{J_n|_N}(\bfx_n)$ for any $n\in\N$.\medskip

          \item[(b)] by combining Lemma \ref{basiclem},  the above (a)  and the fact that $K_{r_n}$ is $(2{r_n})$-dense in $K$, there exist a constant $C_1>1$ (independent of $n\in\N$) and an integer $N=N(n)\in\N$  such that
          \begin{eqnarray*}
              \varphi_{J_n|_N}(K)&\subseteq&\varphi_{J_n|_N}\big((K_{r_n})_{2r_n}\big)\\[4pt]
              &\subseteq&\big(\varphi_{J_{n|N}}(K_{r_n})\big)_{\frac{1}{2}C_1r_n\|\varphi'_{J_n|_N}\|}\\[4pt]
              &\subseteq&\Big(B\big(\bfy_n,C_1\|\varphi'_{J_n|_N}\|\big)\Big)_{\frac{1}{2}C_1r_n\|\varphi'_{J_n|_N}\|}\\[4pt]
&\subseteq&\big(\bfy_n+T_{\bfy_n}M_n\big)_{C_1r_n\|\varphi_{J_n|_N}'\|}\,.
          \end{eqnarray*}
     \end{itemize}
     For any $n\in\N$, define $$\psi_n(\bfz):=\bfy_n+\|\varphi_{J_n|_N}'\|^{-1}(\bfz-\bfy_n),\qquad\forall \ \bfz\in\R^d.$$
    Since $\psi_n$ fixes $\bfy_n$ and it is a similarity mapping, it follows that $\psi_n$ fixes the affine hyperplane $\bfy_n+T_{\bfy_n}M_n$. This together with the above observation (b) implies that
     \begin{equation}\label{psijnphijnk}
         \psi_n\circ\varphi_{J_n|N}(K)\subseteq\big(\bfy_n+T_{\bfy_n}M_n\big)_{C_1r_n}\qquad\forall \ n\in\N.
     \end{equation}
     Note that
     \begin{itemize}
         \item[(c)] the sequence of maps $\{\psi_n\circ\varphi_{J_n|_N}\}_{n\in\N}$ is uniformly bounded on $U$ and there exists $C_2>1$ such that
         \begin{equation*}
             C_2^{-1}|\bfx-\bfy|\leq|\psi_n\circ\varphi_{J_n|_N}(\bfx)-\psi_n\circ\varphi_{J_n|_N}(\bfy)|\leq C_2|\bfx-\bfy|
         \end{equation*}
         for any $\bfx,\bfy\in U$, where $U$ is the open set associated with \eqref{uniformbilip}. Then by \cite[Lemma 2.1]{huang2025}, we may assume without loss of generality that $\psi_n\circ\varphi_{J_n|_N}\to f$ uniformly on $U$ for some conformal map $f:U\to\R^d$.
         \medskip

         \item[(d)] the sequence $\{\bfy_n+T_{\bfy_n}M_n\}_{n\in\N}$ is bounded in $A(d,d-1)$. Since $A(d,d-1)$ is locally compact, by passing to a subsequence if necessary, we may assume that $\bfy_n+T_{\bfy_n}M_n\to V$ for some $V\in A(d,d-1)$.\medskip 
     \end{itemize}
     Let $f:U\to\R^d$ be the conformal map as in (c) and let $V\in A(d,d-1)$ be as in (d). Then, by \eqref{psijnphijnk},  we have $f(K)\subseteq(V)_{\epsilon}$ for any $\epsilon>0$ and thus $f(K)\subseteq V$. It follows that 
     \begin{equation}\label{containana}
         K\subseteq f^{-1}(V\cap f(U)).
     \end{equation}
     Recall (see the Remark~\ref{conforremark}) that any conformal map on an open set in $\R^d$ is real-analytic in the case $d\geq2$. Therefore, the inclusion \eqref{containana} yields a contradiction that $K$ is contained in a $(d-1)$-dimensional real-analytic submanifold. Consequently, there exists $r_2>0$ such that $K_r$ is irreducible and thus hyperplane diffuse for any $0<r<r_2$.
 % As noted in Remark~\ref{irrerek}, $K_r$ is hyperplane diffuse for such $r$.
     \medskip

     From the above two cases we conclude that (i) is true for $K_r$ when $0<r<r_3$, where $r_3:=\min\{r_1,r_2\}$.
     \medskip
     % Let $r_2>0$ such that $(K)_{2r_2}\subseteq\Omega$ and fix $0<r<r_2$. Suppose in contradiction that $K_r$ is not irreducible, then by Remark \ref{irrerek} (ii), $K_r$ is contained in a finite disjoint union of analytic curves, say $\Gamma_1$, $\Gamma_2$,..., $\Gamma_k$. For any $i=1,2,...,k$, write $\Gamma_i=f_i([0,1]\times\{0\})$, where $f_i:\cO_i\to\R^2$ is a conformal map on an open set $\cO_i\supseteq[0,1]\times\{0\}$.  By definition, each $\Gamma_i$ $(i=1,2,...,k)$ is closed. Since they are disjoint, it follows that
     % \begin{equation*}
     %     \min_{i\neq j\in\{1,2,...,k\}}\dist(\Gamma_i,\Gamma_j)>0.
     % \end{equation*}
     % This guarantees the existence of  $J\in\cI_r^*$ and $i_0\in\{1,2,...,k\}$ such that 
     % \begin{equation*}
     %     \varphi_J(K_r)\subseteq \Gamma_{i_0},\qquad \varphi_J(K)\subseteq\varphi_J((K_r)_{2r})\subseteq f_{i_0}(\cO_{i_0}).
     % \end{equation*}

     \noindent $\bullet$ \emph{Proving that for any $0<\epsilon<\dimH K$ there exists $r_0=r_0(\epsilon)>0$ such that $(ii)$ is true for $K_r$ whenever $0<r<r_0$.}  The statement is implicit in the proof of Theorem 4 in Falconer's paper \cite{MR969315}. For the sake of completeness, we include a proof here. Let $0<\epsilon<\dimH K=:s$. Then $s-\epsilon<\underline\dim_{\mathrm{B}} K$ since $\dimH K\leq\underline\dim_{\mathrm{B}} K$, here $\underline\dim_{\mathrm{B}} E$ denotes the lower box dimension of the bounded set $E\subseteq\R^d$; see Section 3.1 of \cite{MR2118797} for its definition. Let $C_3,C_4,C_5>1$ be the constants associated with \eqref{boundeddis}-\eqref{diamofki} respectively. By the construction of $\cI_r$ and the definition of $\underline\dim_{\mathrm{B}} K$, there exists $r_0=r_0(\epsilon)>0$ such that
     \begin{equation}\label{lowbdofcir}
         \#\cI_r>\Big(C_3\cdot C_5\cdot \big(\min_{1\leq i\leq m}\|\varphi_i'\| \big)^{-1}\Big)^{s-\epsilon}\cdot r^{-(s-\epsilon)}
     \end{equation}
     for any $0<r<r_0$. Fix $0<r<r_0$ and define a Borel probability measure $\mu$ on $K_r$ such that $\mu(\varphi_J(K_r))=(\#\cI_r)^{-n}$ for any $n\in\N$ and $J\in\cI_r^n$. This measure is well-defined since $\{\varphi_I\}_{I\in\cI_r}$ satisfies the strong separation condition. Let $$\delta:=\inf\{|\varphi_{I_1}(\bfx)-\varphi_{I_2}(\bfy)|:I_1\neq I_2\in\cI_r,\,\bfx,\bfy\in K_r\}>0.$$
     Then, for any $n\in\N$, $J_1\neq J_2\in\cI_r^n$ and $\bfx,\bfy\in K_r$, by \eqref{uniformbilip} and \eqref{diamofki}, we have
     \begin{equation}\label{DISTJKr}
        |\varphi_{J_1}(\bfx)-\varphi_{J_2}(\bfy)|\geq C_3^{-1}C_4^{-1}\delta\cdot \big(C_3^{-1}C_5^{-1}\min_{1\leq i\leq m}\|\varphi_i'\|\cdot r\big)^n=:c_1\cdot (c_2\cdot r)^n,
     \end{equation}
     where $c_1:=C_3^{-1}C_4^{-1}\delta$ and $c_2:=C_3^{-1}C_5^{-1}\min_{1\leq i\leq m}\|\varphi_i'\|$. Let $n\in\N$ and let $B$ be a subset in $\R^d$  such that
     \begin{equation*}
         c_1\cdot(c_2\cdot r)^{n+1}<|B|\leq c_1\cdot(c_2\cdot r)^{n}.
     \end{equation*}
     By \eqref{DISTJKr}, the subset $B$  intersects at most one element in $\{\varphi_J(K_r)\}_{J\in\cI_r^n}$. This together with \eqref{lowbdofcir} implies that
     \begin{equation*}
         \mu(B)\leq\frac{1}{(\#\cI_r)^n}\leq c_2^{n(s-\epsilon)}\cdot r^{n(s-\epsilon)}\leq (c_1c_2r)^{-(s-\epsilon)}\cdot|B|^{s-\epsilon}.
     \end{equation*}
     It follows that $\dimH K_r\geq s-\epsilon$ by Mass Distribution Principle (see Section 4.1 of \cite{MR2118797}). Since $\Phi_r$ satisfies the strong separation condition, $\lodimA K_r=\dimH K\geq s-\epsilon$.
     \medskip

     The above discussions show that for any $\epsilon>0$, the statement (i) and (ii) of the lemma is satisfied simultaneously for $K_r$ when $r$ is sufficiently small. The proof is thereby complete. 
\end{proof}

\section{Basic properties of Uniformly  Piecewise $C^{1+\kappa}$ expanding maps}\label{UPCEM}

The following result extends the well-known bounded distortion  properties of piecewise $C^{1+\kappa}$  expanding maps on (see \cite[Section 3.1]{mane2012} ) to  the corresponding non-autonomous systems.

\begin{lemma}\label{bounddis}
Let $\cT=\{T_n\}_{n\in\N}$ be a sequence of uniformly piecewise $C^{1+\kappa}$  expanding maps on $[0,1]$.    Then there exists $C>1$ such that
    \begin{eqnarray}\label{bdddisnonat}
        C^{-1}\leq\frac{\big|\big(\sfT_{1,n}\big)'(x)\big|}{\big|\big(\sfT_{1,n}\big)'(y)\big|}\leq C,\quad\forall\ n\in\N,\ \forall\ \bfu\in\Sigma(1,n),\ \forall\ x,y\in I_{\bfu}(1,n).
    \end{eqnarray}
\end{lemma}
\begin{proof}
Recall that $|T_n'(x)|\geq1$ for all $n\in\N$ and $x\in\bigcup_{i\in\Sigma_n}I_i(n)$. Throughout, fix $n\in\N$, $\bfu\in\Sigma(1,n)$ and $x,y\in I_{\bfu}(1,n)$.  Then, by \eqref{c1kappalln} and the inequality $|\log a-\log b|\leq |a-b|$ for all $a,b\geq1$, there exist $C_1>0$ and $\kappa>0$ such that
\begin{eqnarray}
\left|\log\frac{\big|\big(\sfT_{1,n}\big)'(x)\big|}{\big|\big(\sfT_{1,n}\big)'(y)\big|}\right|&\leq&\sum_{i=1}^n\Big|\log\big|T_{i}'\big(\sfT_{1,i-1}(x)\big)\big|-\log\big|T_{i}'\big(\sfT_{1,i-1}(y)\big)\big|\Big|\nonumber\\
&\leq&\sum_{i=1}^n\Big|\big|T_{i}'\big(\sfT_{1,i-1}(x)\big)\big|-\big|T_{i}'\big(\sfT_{1,i-1}(y)\big)\big|\Big|\nonumber\\
&\leq&\sum_{i=1}^n\big|T_{i}'\big(\sfT_{1,i-1}(x)\big)-T_{i}'\big(\sfT_{1,i-1}(y)\big)\big|\nonumber\\
&\leq& C_1\sum_{i=1}^n|\sfT_{1,i-1}(x)-\sfT_{1,i-1}(y)|^{\kappa}\nonumber\\
&\leq& C_1\sum_{i=1}^n\sup\big\{|I_{\bfv}(i,n-i+1)|^{\kappa}:\bfv\in\Sigma(i,n-i+1)\big\}.\label{logt1nxt1ny}
\end{eqnarray}
By the uniformly expanding property \eqref{unifoexpad}, there exist $C_2>0$ and $\rho\in(0,1)$ such that
\begin{equation*}
    |I_{\bfv}(i,n)|\leq C_2\rho^n,\qquad \forall \ i\in\N,\ \forall\ n\in\N, \ \forall \ \bfv\in\Sigma(i,n).
\end{equation*}
On combining this with \eqref{logt1nxt1ny}, we have
\begin{eqnarray*}
    \left|\log\frac{\big|\big(\sfT_{1,n}\big)'(x)\big|}{\big|\big(\sfT_{1,n}\big)'(y)\big|}\right|\leq C_1C_2\sum_{i=1}^n\rho^{i}\leq\frac{C_1C_2\rho}{1-\rho}<+\infty
\end{eqnarray*}
and thus
\begin{eqnarray*}
    \frac{\big|\big(\sfT_{1,n}\big)'(x)\big|}{\big|\big(\sfT_{1,n}\big)'(y)\big|}\leq e^{\frac{C_1C_2\rho}{1-\rho}}.
\end{eqnarray*}
Exchanging the place of $x$ and $y$ in the above process gives the corresponding lower bound. This completes the proof.
\end{proof}
\vspace{1ex}

In the case $T_n(I_i(n))=(0,1)$ for all $n\in\N$ and $i\in\Sigma_n$, the following  result relates the diameter of $I_{\bfu}(i,N)$ to the derivative of the map $\sfT_{i,N}$ on the interval $I_{\bfu}(i,N)$.

\begin{lemma}\label{bdlem}
Let $\cT=\{T_n\}_{n\in\N}$ be a sequence of uniformly piecewise $C^{1+\kappa}$  expanding maps on $[0,1]$  such that $T_n(I_k(n))=(0,1)$ for all $n\in\N$ and $k\in\Sigma_n$. Let $C>1$ be the constant as in Lemma \ref{bounddis}. Then
\begin{eqnarray}\label{diamofiu}
        C^{-1}\leq|I_{\bfu}(i,N)|\cdot \big|\big(\sfT_{i,N}\big)'(x)\big| \leq C,\quad\forall\ i,N\in\N,\  \bfu\in\Sigma(i,N),\ \forall\ x\in I_{\bfu}(i,N).
    \end{eqnarray}
    As a consequence,  we have
    \begin{eqnarray}\label{diaeofiuiv}
     C^{-3}\leq\frac{|I_{\bfu\bfv}(i,N_1+N_2)|}{|I_{\bfu}(i,N_1)|\cdot|I_{\bfv}(i+N_1,N_2)|}\leq C^3
    \end{eqnarray}
    for any $i,N_1,N_2\in\N$, any $\bfu\in\Sigma(i,N_1)$ and any $\bfv\in\Sigma(i+N_1,N_2)$.
\end{lemma}
\begin{proof}
    Fix $i,N\in\N$, $\bfu\in\Sigma(i,N)$ and $x\in I_{\bfu}(i,N)$.  The assumption $T_n(I_k(n))=(0,1)$ for all $n\in\N$ and $k\in\Sigma_n$ ensures the existence of a monotonic $C^{1+\kappa}$ map $f$ on $[0,1]$ such that $f|_{(0,1)}$ is the inverse of $\sfT_{i,N}|_{I_{\bfu}(i,N)}$. Let  $C>1$ be the constant as in Lemma \ref{bounddis}. Then, by the Lemma \ref{bounddis} and the Mean Value Theorem,
    \begin{eqnarray*}
    \begin{aligned}
        |I_{\bfu}(i,N)|&=|f(1)-f(0)|
        \geq\inf_{y\in[0,1]}|f'(y)|\\
        &=\Big(\sup_{y\in I_{\bfu}(i,N)}\big|\big(\sfT_{i,N}\big)'(y)\Big)^{-1}
        \geq\big(C\cdot\big|\big(\sfT_{i,N}\big)'(x)\big|\big)^{-1}.
    \end{aligned}
    \end{eqnarray*}
    This gives the lower bound of the inequality \eqref{diamofiu}. The corresponding upper bound is derived in a similar way. In the end, the inequality \eqref{diaeofiuiv} follows directly from the Lemma \ref{bounddis}, the inequality \eqref{diamofiu} and the chain rule.
\end{proof}

\section{Proof of Theorem \ref{onedimnonaut}}\label{pfofthemainnon}

This section focuses on proving Theorem \ref{onedimnonaut}. Before doing this, we first establish  a strategy for Alice to let Bob avoid given hyperplanes while playing the $\gamma$-hyperplane absolute game. For any subsets $A,B\subset\R^d$, define the distance between $A$ and $B$ as
 \begin{equation*}
     \dist(A,B):=\inf\left\{\|\bfx-\bfy\|:\bfx\in A,\,\bfy\in B\right\},
 \end{equation*}
where $\|\cdot\|$ denotes the maximal norm in $\R^d$.

\begin{lemma}\label{keystr}
    Fix $\gamma\in(0,1/3)$ and let Bob and Alice play $\gamma$-hyperplane absolute game on $[0,1]^d$. Let $\epsilon>0$ be the constant depending on $\gamma$ defined as
    \begin{eqnarray}\label{defofepsi}
         \epsilon=\epsilon(\gamma):=1-\frac{\gamma}{4+6\gamma}\in(0,1)\,.
    \end{eqnarray}
    Let $k\in\N$ and suppose that Bob chooses a closed ball $B_k\subseteq[0,1]^d$ at the $k$-th step. Then for any positive integer $N\in\N$, $j_0\in\{1,...,d\}$ and $y_1,...,y_N\in[0,1]$,  Alice has a strategy to ensure that
    \begin{equation}\label{distpro}
         \#\left\{1\leq i\leq N:\dist(\cL_i,B_{k+1})>\frac{\gamma}{4}|B_k|\right\}\geq\uinte{(1-\epsilon)N}
    \end{equation}
    whatever Bob chooses $B_{k+1}$ at the $(k+1)$-th step, where $\cL_i$'s are the hyperplanes
    \begin{eqnarray}\label{defofcli}
           \cL_i:=\left\{\bfx=(x_1,...,x_d)\in\R^d:x_{j_0}=y_i\right\}\qquad\forall\ i=1,2,...,N.
    \end{eqnarray}
\end{lemma}

% \begin{lemma}\label{keystr}
%     Let $\gamma\in(0,1/3)$ and let Bob and Alice play $\gamma$-absolute game on $[0,1]$. Let
%    \begin{eqnarray}\label{defofepsi}
%     \epsilon=\epsilon(\gamma)=1-\frac{\gamma}{4+6\gamma}\in(0,1).
% \end{eqnarray}
% Let $k\in\N$ and suppose that Bob select a closed ball $B_k\subseteq[0,1]$ at the $k$-th step.
%     Then  for any $N\in\N$ and any points $y_1,y_2,\cdot\cdot\cdot,y_N\in[0,1]$,   Alice has a strategy to ensure that
%     \begin{eqnarray}\label{distpro}
%         \#\left\{1\leq i\leq N: \dist(y_i,B_{k+1})>\frac{\gamma}{2}|B_k| \right\}\geq\uinte{(1-\epsilon) N}
%     \end{eqnarray}
%     whatever Bob chooses $B_{k+1}$ at the $(k+1)$-th step.
% \end{lemma}

\begin{proof}
    We only prove the lemma for the case $j_0=1$;  other cases are similar.  Suppose that Bob chooses a closed ball  $B_k$ at the $k$-th step of $\gamma$-hyperplane absolute game on $[0,1]^N$  and let $y_1,...,y_N\in[0,1]$. In the following, we will find a strategy for Alice to guarantee the validity of \eqref{distpro}.
    % Write
    % \begin{equation*}
    %     B_k=[a_1,b_1]\times\cdot\cdot\cdot\times[a_d,b_d].
    % \end{equation*}
    We divide the proof into the following cases.
    \medskip

    \emph{Case 1: $\#\left\{1\leq i\leq N: \dist(\cL_i,B_{k})>\gamma|B_k|/2\right\}\geq \uinte{N/2}.$}
In this case, \eqref{distpro} is satisfied automatically with $\epsilon=1/2$.
\medskip

\emph{Case 2: $\#\left\{1\leq i\leq N: \dist(\cL_i,B_{k})>\gamma|B_k|/2\right\}\leq\linte{N/2}.$} Write
\[
    B_k=[a_1,b_1]\times\cdot\cdot\cdot\times[a_d,b_d],
\]
where $b_i-a_i=|B_k|$ for any $i=1,...,d$. Then this case is  equivalent to 
\begin{eqnarray}\label{numofyi}
    \#\left\{1\leq i\leq N:y_i\in\Big[a_1-\frac{\gamma|B_k|}{2},b_1+\frac{\gamma|B_k|}{2}\Big]\right\}\geq\uinte{N/2}.
\end{eqnarray}
  Partition the interval $\Big[a_1-\frac{\gamma|B_k|}{2},b_1+\frac{\gamma|B_k|}{2}\Big]$ into $\uinte{(1+\gamma)/(\gamma/2)}$  closed subintervals with length at most $\gamma|B_k|/2$, then the inequality \eqref{numofyi} implies that one of these intervals contains  $$\frac{\uinte{N/2}}{\uinte{(1+\gamma)/(\gamma/2)}}$$ or more of the points $y_i$. In other words, there exists a closed subinterval $$I\subseteq\Big[a_1-\frac{\gamma|B_k|}{2},b_1+\frac{\gamma|B_k|}{2}\Big],\qquad |I|=\gamma|B_k|/2$$  such that
\begin{eqnarray}\label{numyinI}
\begin{aligned}
    \#\left\{1\leq i\leq N:y_i\in I\right\}&\geq \Big\lceil{\frac{\uinte{N/2}}{\uinte{(1+\gamma)/(\gamma/2)}}}\Big\rceil\geq\uinte{\frac{N/2}{\frac{1+\gamma}{\gamma/2}+1}}\\
    &=\uinte{\frac{\gamma\cdot N}{4+6\gamma}}=\uinte{(1-\epsilon)N},
\end{aligned}
\end{eqnarray}
where $\epsilon$ is defined as  in \eqref{defofepsi}.
Let $z\in I$ be the center of this interval and let $\cL$ be the hyperplane $\{\bfx\in\R^d:x_1=z\}$. It follows from \eqref{numyinI} that
\begin{eqnarray}\label{numcliincl}
    \#\{1\leq i\leq N:\cL_i\subseteq(\cL)_{\frac{\gamma|B_k|}{4}}\}=\#\{1\leq i\leq N:y_i\in I\}\geq\uinte{(1-\epsilon)N}.
\end{eqnarray}
Now, let Alice choose $A_k=(\cL)_{\frac{\gamma|B_k|}{2}}$. Then by \eqref{numcliincl}, we have 
\[
\#\left\{1\leq i\leq N: \dist(\cL_i,B)>\frac{\gamma}{4}|B_k| \right\}\geq\#\left\{1\leq i\leq N:\cL_i\in (\cL)_{\frac{\gamma|B_k|}{4}}\right\}\geq\uinte{(1-\epsilon)N}
\]
for any closed ball $B\subseteq B_k\setminus A_k$. This implies that the inequality \eqref{distpro} is satisfied 
whatever Bob chooses $B_{k+1}$ at the $(k+1)$-th step.
\medskip

On combining Case 1 and Case 2 with the easily verified inequality $\gamma/(4+6\gamma)<1/2$, we have proved that the inequality \eqref{distpro} is valid with $\epsilon$  defined as  in \eqref{defofepsi}.
\end{proof}

The following direct corollary of Lemma \ref{keystr}  is more convenient to use in  proving Theorem \ref{onedimnonaut}.

\begin{corollary}\label{corokeystra}
     Let $\gamma\in(0,1/3)$ and let Bob and Alice play $\gamma$-hyperplane absolute game on $[0,1]^d$.   Let $k\in\N$ and suppose that Bob selects a closed ball $B_k\subseteq[0,1]^d$ at the $k$-th step. Then
     \begin{itemize}
         \item[(i)] for any $y\in[0,1]$ and $j_0\in\{1,2,...,d\}$, Alice has a strategy to ensure that $$\{\bfx\in\R^d:x_{j_0}= y\}\cap B_{k+1}=\emptyset$$ whatever Bob chooses $B_{k+1}$ at the $(k+1)$-th step.
         \medskip

         \item[(ii)] for any $N\in\N$, any $j_0\in\{1,2,...,d\}$ and any intervals $I_1,I_2,...,I_N\subseteq[0,1]$ with length not greater than $\gamma^s|B_k|/2$, where $s=s(\gamma,N):=\linte{\log_{1/\epsilon}N}+1$ and $\epsilon=\epsilon(\gamma)$ is defined as in \eqref{defofepsi}, Alice has a strategy to ensure that
         \[
             B_{k+s}\ \cap \ \bigcup_{i=1}^N\,\{\bfx\in\R^d:x_{j_0}\in I_i\}=\emptyset
         \]
         whatever Bob chooses $B_{k+s}$ at the $(k+s)$-th step.
     \end{itemize}
\end{corollary}
\begin{proof}
    The statement (i) follows by directly applying Lemma \ref{keystr} to the  point $y$. The statement (ii) is obtained by repeatedly applying Lemma \ref{keystr}  to the points $y_1,...,y_N$ for $s$ times.
\end{proof}

We reiterate that in the setting of Theorem \ref{onedimnonaut}, $\cT=\{T_{n}\}_{n\in\N}=\{T_{n,1}\otimes T_{n_2}\}_{n\in\N}$ is a sequence of self-maps on $[0,1]^d$ such that $\cT_1=\{T_{n,1}\}_{n\in\N}$ is a sequence of uniformly $C^{1+\kappa}$ piecewise expanding maps on $[0,1]$ satisfying either assumption $(\mathrm{A}')$ or assumption $(\mathrm{B}')$ stated above Theorem \ref{onedimnonaut} and $\cT_2=\{T_{n,2}\}_{n\in\N}$ is a sequence of self-maps on $[0,1]^{d-1}$, and $\cG=\{g_n\}_{n\in\N}$ is a sequence of self-maps on $[0,1]^d$ that satisfy the uniform Lipschtz condition  \eqref{conslipgn}.  For any $n\in\N$, let $\Sigma_n$ and $\{I_i(n)\}_{i\in\Sigma_n}$ be the associated alphabet and disjoint open intervals associated with $T_{n,1}$ as stated in the Definition \ref{defofpiee} (i).  In turn, for any  $i\in\N$ and $N\in\N$, let $\Sigma(i,N)$ be defined as in \eqref{SigmaiN} and let $I_{\bfu}(i,N)$ be defined as in \eqref{IuiN} for any $\bfu\in\Sigma(i,N)$, with $\cT$ replaced by $\cT_1$. For any $i,N\in\N$, write
\begin{eqnarray*}
     \sfT_{i,N}&:=&T_{i+N-1}\circ T_{i+N-2}\circ\cdot\cdot\cdot\circ T_{i}\,,\\
     \sfT_{i,N,k}&:=&T_{i+N-1,k}\circ T_{i+N-2,k}\circ\cdot\cdot\cdot\circ T_{i,k}\,,\quad k=1,2.
\end{eqnarray*}
 Now, for any $\delta>0$ and any $n\in\N$, define
\begin{eqnarray*}
    \tN(\cT,\cG,\delta,n):=\left\{\bfx\in[0,1]^d:\|\sfT_{1,n}(\bfx)-g_n(\bfx)\|>\delta\right\}.
\end{eqnarray*}
Then the sets $\tN(\cT,\cG)$ and $\tN(\cT,\cG,\delta,n)$ are related via
$$\tN(\cT,\cG)=\bigcup_{\delta>0}\bigcup_{N=1}^{\infty}\bigcap_{n=N}^{\infty}\tN(\cT,\cG,\delta,n).$$
Throughout, for any ball $B\subseteq\R^d$ associated with the maximal norm, write
\[
    B=B_1\times B_2\times\cdot\cdot\cdot B_d,
\]
where $B_i\subseteq\R$ is an interval with length $|B|$ for any $i=1,2,...,d$. 
The next result concerns the local structure of the set $\tN(\cT,\cG,\delta,n)$.

\begin{lemma}\label{lemmaconint}
   Let $\cT$, $\cT_1$ and $\cG$ be defined as above. Let $C>0$ be the constant as in \eqref{conslipgn}. %Suppose that  $n\in\N$ and $\bfu\in\Sigma(1,n)$ satisfy
    % \begin{eqnarray}\label{lowder}
    %    \inf_{x\in I_{\bfu}(1,n)}\big|\big(\sfT_{\cF_1}(1,n)\big)'(x)\big|>C\,.
    % \end{eqnarray}
    Then, for any $n\in\N$, $\bfu\in\Sigma(1,n)$, $\delta>0$ and any closed ball $B\subseteq I_{\bfu}(1,n)\times[0,1]^{d-1}$, there exists an interval $J\subseteq\R$ with length
    $$\frac{2\delta+C|B|}{\inf_{x\in B_1}\big|\sfT_{1,n,1}'(x)\big|}$$ 
such that 
    \begin{eqnarray}\label{IcapN}
        B\cap \tN(\cT,\cG,\delta,n)^c\subseteq\{\bfx\in\R^d:x_1\in J\}.
    \end{eqnarray}
\end{lemma}
\begin{proof}
   % Write $B=I_1\times I_2\times\cdot\cdot\cdot I_d$, where $I_j\subseteq[0,1]$ is an interval with length $|B|$ for any $j=1,2,...,d$.
   Let $\bfx,\bfy$ be two points belonging to 
 the set in the left-hand-side of  \eqref{IcapN}. Then by Mean Value Theorem, triangle inequality and \eqref{unilip}, we have
\begin{eqnarray*}
    2\delta&\geq&\big\|(\sfT_{1,n}(\bfx)-g_n(\bfx))-(\sfT_{1,n}(\bfy)-g_n(\bfy))\big\|\\[4pt]
    &\geq&\|\sfT_{1,n}(\bfx)-\sfT_{1,n}(\bfy)\|-\|g_n(\bfx)-g_n(\bfy)\|\\[4pt]
    &\geq&|\sfT_{1,n,1}(x_1)-\sfT_{1,n,1}(y_1)|-C|B|\\[4pt]
    &\geq&\inf_{z\in B_1}\big|\sfT_{1,n,1}'(z)\big|\cdot|x_1-y_1|-C|B|.
\end{eqnarray*}
It follows that
\[
   |x_1-y_1|\leq \frac{2\delta+C|B|}{\inf_{x\in B_1}\big|\sfT_{1,n,1}'(x)\big|}\,,\quad \forall\ \bfx,\bfy\in B\cap \tN(\cT,\cG,\delta,n)^c
\]
Thereby the proof is complete.
\end{proof}

Now we are going to prove Theorem \ref{onedimnonaut}  by considering separately assumption ($\mathrm{A}'$) and assumption ($\mathrm{B}'$) appearing in the statement of the theorem.
\subsection{Proof of Theorem \ref{onedimnonaut}: assuming ($\mathrm{A}'$)}Throughout the subsection, let us fix $\gamma\in(0,1/3)$. Our aim is proving that $\tN(\cT,\cG)$ is $\gamma$-hyperplane absolute winning on $[0,1]^d$. Before doing this, we introduce  some constants as follows:
\begin{itemize}
 \item $M:=\sup\left\{|T_{n,1}'(x)|:n\in\N,\,x\in\bigcup_{i\in\Sigma_n}I_i(n)\right\}$. By the assumption ($\mathrm{A}'$), we have $$\#\{I_i(n):n\in\N,i\in\Sigma_n\}<+\infty.$$  This together with  \eqref{c1kappalln} shows that  $M<+\infty$.
    \medskip

    \item $\epsilon=\epsilon(\gamma)\in(0,1)$ is the constant as in \eqref{defofepsi}.
    \medskip

    \item $C_1>0$ stands for the constant as in \eqref{conslipgn} associated with $\cG$.
    \medskip

    \item $C_2>1$ denotes the constant as in \eqref{bdddisnonat} associated with $\cT_1$.
    \medskip

    \item Fix $N\in\N$ so large such that
    \begin{eqnarray}\label{defofNA}
    \inf\Big\{\big|\sfT_{i,N,1}'(x)\big|:i\in\N,\,x\in\bigcup_{\bfu\in\Sigma(i,N)}I_{\bfu}(i,N)\Big\}>C_2M\cdot\gamma^{-s}.
    % \left\{
    % \begin{aligned}
    %     &\inf\Big\{\big|\big(\sfT_{\cF_1}(i,N)\big)'(x)\big|:i\in\N,\,x\in\bigcup_{\bfu\in\Sigma(i,N)}I_{\bfu}(i,N)\Big\}>C_2M\cdot\gamma^{-s},\\
    %     &\gamma^{-s}>2MC_1C_2,
    % \end{aligned}
    % \right.
    \end{eqnarray}
    Here $s:=2+s_1+s_2$, where  $s_1:=\linte{\log_{1/\epsilon}(N+1)}+1$ and $s_2\in\N$ satisfies 
    \begin{equation}\label{defofs2}
        \gamma^{s_2}<(2\gamma+C_1)^{-1}C_2^{-1}M^{-1}.
    \end{equation}
   The existence of $N$ follows from the inequality \eqref{unifoexpad}. 
\medskip

\item Let $\ell:=\min\left\{|I_{\bfu}(i,N)|:i\in\N,\,\bfu\in\Sigma(i,N),\,I_{\bfu}(i,N)\neq\emptyset\right\}$, where $N$ is the integer defined as above. The minimum is reached because the number of elements of the set in the definition of $\ell$  is finite while assuming ($\mathrm{A}'$). Also note that $|I_{\bfu}(i,N)|>0$ if $I_{\bfu}(i,N)\neq\emptyset$ since in this case $I_{\bfu}(i,N)$ is a non-empty open interval. It follows that $\ell>0$.
    \medskip

    \item Let $\rho_k$ denote the diameter of $B_k$ for every integer $k\geq1$, where $B_k\subseteq[0,1]^d$ is the closed ball that Bob chooses at the $k$-th step. Naturally, $\rho_k$ depends on Bob's strategy. Also,  the sequences $\{n_k\}_{k\in\N}$ and $\{m_k\}_{k\in\N}$ defined below are dependent on Bob's choices of balls.
    \medskip

    \item Define $n_0:=2$ and $n_k:=\inf\left\{n\in\N:\rho_n<\gamma^{ks}\rho_2\right\}$ for any integer 
 $k\geq1$. Here we adopt the convention that $\inf\emptyset=+\infty$. 
 \medskip

%  \item  For any integer $n\geq0$, write
% \begin{equation}
%     B_{n}=B_{n,1}\times B_{n,2}\times\cdot\cdot\cdot B_{n,d},
% \end{equation}
% where $B_{n,i}$ is a closed interval with length $|B_{n}|$ on $[0,1]$ for each $i=1,2,...,d$. 
\item Given an integer $k\geq0$,
 if $n_k<+\infty$, then let $$m_k:=\inf\left\{n\in\N:\sup_{x\in B_{n_k,1}\setminus E_n}\big|\sfT_{1,n,1}'(x)\big|>M^{-1}\gamma^{-(k+1)s}\right\},$$ where $E_n:=\bigcup_{\bfu\in\Sigma(1,n)}\partial I_{\bfu}(1,n)$. Here, $\partial E$ denotes the boundary of  $E$.
\medskip
   
    \item $\delta:=\gamma\rho_2$.
\end{itemize}
\medskip

Without loss of generality, throughout we assume that
\[
    \rho_k\to0\ (k\to\infty)\quad\text{and}\quad \rho_1<\ell\cdot\min\{\gamma^s,M^{-1}\}.
\]
 Then it follows that $n_k<+\infty$ for any $k\geq0$ and that $n_k\to\infty$ as $k\to\infty$. With this and inequality \eqref{unifoexpad} in mind, $m_k$ is well-defined and finite for any $k\geq0$. Moreover, it is easily verified that the sequence $\{m_k\}_{k\geq0}$ is increasing monotonically to infinity. In order to prove that $\tN(\cT,\cG)$ is $\gamma$-hyperplane absolute winning,   we make the subsequent claim.

\begin{claim}\label{finitecla}
    In the $\gamma$-hyperplane absolute game on $[0,1]^d$, Alice has a strategy to ensure that for any $k\geq0$, the following statements hold:
\begin{itemize}
    \item[(i)] There exist $\bfu\in\Sigma(1,m_k)$ and $\bfv\in\Sigma(1,m_k+N)$ such that 
    \begin{equation}\label{balinclucy}
        B_{n_k}\subseteq I_{\bfu}(1,m_k)\times [0,1]^{d-1}\quad\text{and}\quad B_{n_{k}+1}\subseteq I_{\bfv}(1,m_k+N)\times [0,1]^{d-1}.\medskip
    \end{equation}

    \item[(ii)] For any $m_k\leq n\leq m_k+N$, we have $B_{n_k+s}\subseteq \tN(\cT,\cG,\delta,n)$.
    %\medskip

    %\item[(iii)] $0\leq m_{k+1}-m_k\leq N$.
\end{itemize}
\end{claim}

\noindent    The claim implies that $\tN(\cT,\cG)$ is $\gamma$-absolutely winning on $[0,1]^d$. To see this, we first bound the gap between $m_k$ and $m_{k+1}$.

\begin{lemma}\label{gapmkmk1}
    Let $k\geq0$ be an integer for which the Claim \ref{finitecla} (i) holds, then we have
    \begin{eqnarray}\label{gapbetmkmk1}
        0\leq m_{k+1}-m_k\leq N.
    \end{eqnarray}
\end{lemma}
\begin{proof}
    Fix an integer $k\geq0$ such that (i) of Claim \ref{finitecla} is true for such $k$. Let $\bfu\in\Sigma(1,m_k)$ and $\bfv\in\Sigma(1,m_k+N)$ satisfy \eqref{balinclucy}. The lower bound in \eqref{gapbetmkmk1} follows from the fact that  the sequence $\{m_{k'}\}_{k'\geq0}$ is increasing. We next analyze the upper bound.   With (i) of Claim \ref{finitecla} in mind, the map $\sfT_{1,m_k,1}$ is differentiable on $B_{n_k,1}$, and thus on $B_{n,1}$ for any $n\geq n_k$. This together with  Lemma \ref{bounddis} and the definition of $m_k$ implies that 
    \[
        \inf_{x\in B_{n_{k+1},1}}\big|\sfT_{1,m_k,1}'(x)\big|\geq C_2^{-1}\sup_{x\in B_{n_k,1}}\big|\sfT_{1,m_k,1}'(x)\big|>(C_2M)^{-1}\gamma^{-(k+1)s}.
    \]
    Then, by the chain rule and the first inequality in \eqref{defofNA}, for any $x\in B_{n_{k+1}}$,
    \begin{eqnarray*}
        &~&\hspace{-8ex}\big|\sfT_{1,m_k+N,1}'(x)\big|\\[4pt]
        &~&=\big|\sfT_{m_k+1,N,1}'\big(\sfT_{1,m_k,1}(x)\big)\big|\cdot\big|\sfT_{1,m_k,1}'(x)\big|\\[4pt]
&~&\hspace{0ex}\geq\ \ \inf\Big\{\big|\sfT_{m_k+1,N,1}'(y)\big|:\bfv\in\Sigma(m_k+1,N),\,y\in I_{\bfv}(m_k+1,N)\Big\}\\[4pt]
&~&\hspace{5ex}\times \inf_{z\in B_{n_{k+1},1}}\big|\sfT_{1,m_k,1}'(z)\big|\\[4pt]
        &~&\hspace{0ex} >\ \  \gamma^{-(k+2)s}.
    \end{eqnarray*}
    Therefore, by the definition of $m_{k+1}$, we have $m_{k+1}\leq m_k+N$.
\end{proof} 

We continue  proving, modulo the Claim \ref{finitecla}, that $\tN(\cT,\cG)$ is $\gamma$-hyperplane absolute winning. Let $\bfx\in\bigcap_{k=1}^{\infty}B_k$. By combining Lemma  \ref{gapmkmk1} with the fact that $m_k\to\infty$ ($k\to\infty$), we have
\[
    [m_0,\infty)=\bigcup_{k=0}^{\infty}\,[m_k,m_k+N].
\]
Then, for any integer $n\geq m_0$, there exists $k=k(n)\in\N$ such that $m_k\leq n\leq m_k+N$. On applying  Claim \ref{finitecla} (ii) to $k=k(n)$, we have $\bfx\in\tN(\cT,\cG,\delta,n)$. Since $n\geq m_0$ is arbitrary, it follows that $\bfx\in \tN(\cT,\cG)$. Therefore, the set  $\tN(\cT,\cG)$ is $\gamma$-hyperplane absolute winning.\medskip

Now, it is left to prove  the Claim \ref{finitecla}. For the closed ball $B_k$ that Bob chooses at the $k$-th step, write
\begin{equation*}
    B_k=B_{k,1}\times B_{k,2}\times\cdots B_{k,d},
\end{equation*}
where $B_{k,i}$ is a closed interval in $\R$ for any $1\leq i\leq d$.
\medskip

\noindent\emph{$\bullet$ Proving Claim \ref{finitecla} for $k=0$}: Recall that $$\ell=\min\left\{|I_{\bfu}(i,N)|:i\in\N,\,\bfu\in\Sigma(i,N),\,I_{\bfu}(i,N)\neq\emptyset\right\}.$$ Since $\rho_1<\ell\cdot\min\{\gamma^s,M^{-1}\}<\ell$, $B_{1,1}$ contains at most one endpoint of the intervals $\{I_{\bfu}(1,N):\bfu\in\Sigma(1,N)\}$. 
By (i) of Corollary \ref{corokeystra}, Alice has a strategy such that $B_{2,1}$ contains no endpoint of the intervals $\{I_{\bfu}(1,N):\bfu\in\Sigma(1,N)\}$.  On the other hand, by the definition of $m_0$ and  \eqref{defofNA}, we have $m_0\leq N$. It follows that $B_{2,1}$ is  contained in $I_{\bfu}(1,m_0)$ for some $\bfu\in\Sigma(1,m_0)$. This establishes the first inclusion of \eqref{balinclucy} for $k=0$. %Now $\sfT_{\cF_1}(1,m_0)$ is  $C^1$ and injective on $B_{2,1}$.
To prove the second inclusion, we first establish the inequality $\big|\sfT_{1,m_0,1}(B_{2,1})\big|<\ell$ by separately considering the following cases. Note that $\sfT_{1,m_0,1}$ is monotonic and $C^{1+\kappa}$ on $B_{2,1}$.
\begin{itemize}
    \item[$\circ$] Suppose that $m_0>1$. This together with Mean Value Theorem and the definition of $m_0$ and $M$ shows that
    \begin{equation}\label{m0ge1}
        \begin{aligned}
     &\hspace{-2ex}\big|\sfT_{1,m_0,1}(B_{2,1})\big|\ \leq\ \sup_{x\in B_{2,1}}\big|\sfT_{1,m_0,1}'(x)\big|\cdot |B_{2,1}|\\[4pt]
&\hspace{12ex}\ \leq\ \sup\big\{|T_{m_0,1}'(y)|:i\in\Sigma_{m_0},\,y\in I_i(m_0)\big\}\\[4pt]
&\hspace{20ex}\times\sup_{x\in B_{2,1}}\big|\sfT_{1,m_0-1,1}'(x)\big|\cdot|B_{2,1}|\\[4pt]
    &\hspace{12ex}\ < \ M\cdot M^{-1}\gamma^{-s}\cdot \ell\cdot\min\{\gamma^s,M^{-1}\}\\[4pt]
    &\hspace{12ex}\ \leq\ \ell\,.
\end{aligned}\vspace*{1ex}
    \end{equation}
    
\item[$\circ$] Suppose that $m_0=1$. Then by Mean Value Theorem,
\begin{equation}\label{m0eq1}
    \begin{aligned}
        \big|\sfT_{1,m_0,1}(B_{2,1})\big|&\ =\ |T_{1,1}(B_{2,1})|\ \leq \ \sup_{x\in\bigcup_{i\in\Sigma_1}I_i(1)}|T_{1,1}'(x)|\cdot|B_{2,1}|\\[4pt]
    &\ < \ M\cdot\ell\cdot\min\{\gamma^s,M^{-1}\} \ \leq \ \ell\,.
    \end{aligned}\vspace*{1ex}
\end{equation}
\end{itemize}
\noindent Since $|\sfT_{1,m_0,1}(B_{2,1})|<\ell$, the image $\sfT_{1,m_0,1}(B_{2,1})$ contains at most one endpoint of the intervals $\{I_{\bfw}(m_0+1,N):\bfw\in\Sigma(m_0+1,N)\}$, which implies that $B_{2,1}$ contains at most one endpoint in the intervals $\{I_{\bfv}(1,m_0+N):\bfv\in\Sigma(1,m_0+N)\}$. Using Corollary \ref{corokeystra} (i) again, Alice has a strategy to ensure that $B_{3,1}$ is contained within $I_{\bfv}(1,m_0+N)$ for some $\bfv\in\Sigma(1,m_0+N)$. This proves Claim \ref{finitecla} (i) for $k=0$.
%   Now, by combining the inequality $\big|\sfT_{1,m_0,1}(B_{2,1})\big|<\ell$, the definition of $\ell$, the fact that $B_{2,1}$ is contained in a cylinder set associated with a word in $\Sigma(1,m_0)$ and  that $\sfT_{1,m_0,1}$ is $C^1$ and injective on $B_{2,1}$, we obtain that
% \begin{eqnarray}\label{b2capendcy}
% \begin{aligned}
%     &~\hspace{-6ex}\#\left(B_{2,1}\ \cap\bigcup_{\bfv\in\Sigma(1,m_0+N)}\partial I_{\bfv}(1,m_0+N)\right)\\[4pt]
%     &~\hspace{2ex}=\ \#\left(\sfT_{1,m_0,1}(B_{2,1})\ \cap\bigcup_{\bfw\in\Sigma(m_0+1,N)}\partial I_{\bfw}(m_0+1,N)\right)
%     \ \leq\  1.
% \end{aligned}
% \end{eqnarray}
% Using Corollary \ref{corokeystra} (i) again, Alice can find $A_2$ such that  $$B_3\cap\big(\partial I_{\bfv}(1,m_0+N)\times[0,1]^{d-1}\big)=\emptyset,\quad\forall\ \bfv\in\Sigma(1,m_0+N)$$ whatever Bob chooses $B_3\subseteq B_2\setminus A_2$. This shows that $B_3\subseteq I_{\bfv}(1,m_0+N)\times[0,1]^{d-1}$ for some $\bfv\in\Sigma(1,m_0+N)$, and thus Claim \ref{finitecla} (i)  is proved  for $k=0$.

Next we prove Claim \ref{finitecla} (ii) for $k=0$.  By the claim (i) which we have proved, for any $m_0\leq n\leq m_0+N$, the interval $B_{3,1}$ is contained in a cylinder set associated with a word in  $\Sigma(1,n)$. Then, on combining the definitions of $m_0$ and $\delta$, Lemma \ref{lemmaconint},  Lemma \ref{bounddis} and \eqref{defofs2}, we obtain that there exist $N+1$ closed intervals $J_1,J_2,...,J_{N+1}\subseteq[0,1]$ of length
\begin{eqnarray*}
    \leq\frac{2\delta+C_1|B_3|}{\inf_{x\in B_{3,1}}\big|\sfT_{1,m_0,1}'(x)\big|}\leq\frac{2\delta+C_1|B_3|}{C_2^{-1}M^{-1}\gamma^{-s}} \leq(2\gamma+C_1)C_2M\rho_2\gamma^{s}<\gamma^{s_1}|B_3|
\end{eqnarray*}
such that
\begin{eqnarray}\label{setx1inji}
     \bigcup_{n=m_0}^{m_0+N}B_3\cap \tN(\cT,\cG,\delta,n)^c\subseteq\bigcup_{i=1}^{N+1}\left\{\bfx\in\R^d:x_1\in J_i\right\}.
\end{eqnarray}
% is contained in 
% \begin{eqnarray}
%     \bigcup_{i=1}^{N+1}\left\{\bfx\in\R^d:x_1\in J_i\right\}.
% \end{eqnarray}
By applying Corollary \ref{corokeystra} (ii) to the intervals $J_1,J_2,...,J_{N+1}$, Alice can guarantee  that the ball $B_{3+s_1}$ does not intersect the set in the right-hand-side of \eqref{setx1inji}.  
% and the set of these centers is 
% \[
%     >\frac{\gamma}{2}|B_{2+s_1}|\geq\frac{\gamma^{s}}{2}\rho_2> 2MC_2\delta\gamma^s.
% \]
Since $B_{3+s_1}\supseteq B_{2+s}=B_{n_0+s}$, this implies that $$B_{2+s}\subseteq \tN(\cT,\cG,\delta,n),\qquad\forall \ m_0\leq n\leq m_0+N.$$  The proof of Claim \ref{finitecla} for $k=0$ is complete.
\medskip

\noindent\emph{$\bullet$ If Claim \ref{finitecla} is true for $k=j$, then it is true for $k=j+1$:} The proof is an adaptation of the method  used above. Suppose that Bob has chosen a closed ball $B_{n_j+s}$ that satisfies Claim \ref{finitecla} (ii). Note that $n_{j+1}\geq n_j+s$ by definition. We  first let Alice play arbitrarily until Bob chooses $B_{n_{j+1}}$. From the inductive hypothesis, $B_{n_{j+1},1}$ is contained in a cylinder set associated with a word in $\Sigma(1,m_j+N)$. Moreover, by combining Lemma \ref{gapmkmk1} with the fact that (i) of Claim \ref{finitecla} is true for $k=j$, we have $m_{j+1}\leq m_j+N$. The upshot of the above is that  $B_{n_{j+1}}$ is contained in $I_{\bfu}(1,m_{j+1})\times[0,1]^{d-1}$ for some $\bfu\in\Sigma(1,m_{j+1})$. This proves the first inclusion in \eqref{balinclucy} for $k=j+1$. 

By the same method used in deriving \eqref{m0ge1} and \eqref{m0eq1}, it is possible to show that $\big|\sfT_{1,m_{j+1},1}(B_{n_{j+1},1})\big|<\ell$. Indeed, observe that
\begin{itemize}
    \item[$\circ$] if $m_{j+1}>1$, then on combining Mean Value Theorem and the definition of $m_{j+1}$, we obtain that
    \begin{equation*}\label{mj1ge1}
        \begin{aligned}
     &\hspace{4ex}\big|\sfT_{1,m_{j+1},1}(B_{n_{j+1},1})\big|\ \leq\ \sup_{x\in B_{n_{j+1},1}}\big|\sfT_{1,m_{j+1},1}'(x)\big|\cdot |B_{n_{j+1},1}|\\[4pt]
&\hspace{29ex} \leq\ \sup\big\{|T_{m_{j+1},1}'(y)|:i\in\Sigma_{m_{j+1}},\,y\in I_i(m_{j+1})\big\}\\[4pt]
&\hspace{34ex}\times\sup_{x\in B_{n_{j+1},1}}\big|\sfT_{1,m_{j+1}-1,1}'(x)\big|\cdot|B_{n_{j+1},1}|\\[4pt]
    &\hspace{29ex} < \ M\cdot M^{-1}\gamma^{-(j+2)s}\cdot \gamma^{(j+1)s}\rho_2 
    \\[4pt]
    &\hspace{29ex} \leq\ \ell.
\end{aligned}\vspace*{1ex}
    \end{equation*}
    
\item[$\circ$] if $m_{j+1}=1$, then by Mean Value Theorem,
\begin{equation*}\label{mj1eq1}
    \begin{aligned}
       \hspace{2ex} \big|\sfT_{1,m_{j+1},1}(B_{n_{j+1},1})\big|&\ =\ |T_{1,1}(B_{n_{j+1},1})|\ \leq \ \sup_{x\in\bigcup_{i\in\Sigma_1}I_i(1)}|T_{1,1}'(x)|\cdot|B_{2,1}|\\[4pt]
    &\ < \ M\cdot\ell\cdot\min\{\gamma^s,M^{-1}\} \ \leq \ \ell\,.
    \end{aligned}\vspace*{1ex}
\end{equation*}
\end{itemize}
 Since $|\sfT_{1,m_{j+1},1}(B_{n_{j+1},1})|<\ell$, the interval $B_{n_{j+1},1}$ intersects at most one endpoint of the intervals $\{I_{\bfv}(1,m_{j+1}+N):\bfv\in\Sigma(1,m_{j+1}+N)\}$.
% \[
%     \#\left(B_{n_{j+1},1}\ \cap\bigcup_{\bfv\in\Sigma(1,m_{j+1}+N)}\partial I_{\bfv}(1,m_{j+1}+N)\right)\leq1.
% \]
On combining this with Corollary \ref{corokeystra} (i),  Alice has a strategy to guarantee that 
 $B_{n_{j+1}+1}\subseteq I_{\bfv}(1,m_{j+1}+N)\times[0,1]^{d-1}$ for some $\bfv\in\Sigma(1,m_{j+1}+N)$.   This proves Claim \ref{finitecla} (i) for $k=j+1$. 

Next, we prove Claim \ref{finitecla} (ii) for $k=j+1$. Note that by  Claim \ref{finitecla} (i) for $k=j+1$ that we have just proved,  we conclude that for any $m_{j+1}\leq n\leq m_{j+1}+N$,  the interval $B_{n_{j+1}+1,1}$ is a sub-interval of $I_{\bfv}(1,m_{j+1}+N)$.  Together with the definition of $m_{j+1}$, Lemma \ref{lemmaconint}, Lemma \ref{bounddis} and \eqref{defofs2}, this implies the existence of $N+1$ closed intervals $J_1,J_2,...,J_{N+1}\subseteq[0,1]$ of length
\begin{eqnarray*}
    &\leq& \frac{2\delta+C_1|B_{n_{j+1}+1,1}|}{\inf_{x\in B_{n_{j+1}+1},1}\big|\sfT_{1,m_{j+1},1}'(x)\big|}
    \ \leq\ \frac{2\delta+C_1|B_{n_{j+1}+1,1}|}{C_2^{-1}M^{-1}\gamma^{-(j+2)s}}\\[4pt] 
    &\leq&(2\gamma+C_1)C_2M\rho_2\gamma^{(j+2)s}\ \leq\ \gamma^{s_1}|B_{n_{j+1}+1}|.
\end{eqnarray*}
such that
\begin{eqnarray}\label{nj1x1inji}
    \bigcup_{n=m_{j+1}}^{m_{j+1}+N}B_{n_{j+1}+1}\cap \tN(\cT,\cG,\delta,n)^c\subseteq\bigcup_{i=1}^{N+1}\left\{\bfx\in\R^d:x_1\in J_i\right\}.
\end{eqnarray}
 By applying Corollary \ref{corokeystra} (ii) to the intervals $J_1,J_2,...,J_{N+1}$, Alice has a strategy to ensure that the closed ball $B_{n_{j+1}+1+s_1}$ does not intersect the set in the right-hand-side of \eqref{nj1x1inji}. This shows  that $$B_{n_{j+1}+s}\subseteq \tN(\cT,\cG,\delta,n),\quad\forall\ m_{j+1}\leq n\leq m_{j+1}+N.$$  Now, we have completed the proof of Claim \ref{finitecla} for $k=j+1$.\medskip

The above discussions verify that Claim \ref{finitecla} is true. As mentioned earlier, the claim implies  that the set $\tN(\cT,\cG)$ is $\gamma$-hyperplane absolute winning. By the arbitrariness of $\gamma\in(0,1/3)$,  the set is hyperplane absolute winning.

\subsection{Proof of Theorem \ref{onedimnonaut}: assuming ($\mathrm{B}'$)} 

Let $\gamma\in(0,1/3)$. Before proving that $\tN(\cT,\cG)$ is $\gamma$-hyperplane absolute winning, we first fix some notations as follows:

\begin{itemize}
    \item $\epsilon\in(0,1)$ is the constant as in Lemma \ref{keystr} associated with $\gamma$-hyperplane absolute game.
\medskip

 \item $C_1>0$ is the constant  in \eqref{conslipgn} associated with $\cG$.
    \medskip

    \item Let $C_2>1$ denote the constant as in \eqref{bdddisnonat} and let $C_3=C_2^{\,3}$.
    \medskip
    
    \item $N\in\N$ satisfies that
    \begin{eqnarray}\label{defofN1}
    \begin{aligned}
       \inf\Big\{\big|\sfT_{i,N,1}'(x)\big|:i\in\N,\, x\in \bigcup_{\bfu\in\Sigma(i,N)}I_{\bfu}(i,N)\Big\}>C_3\cdot\gamma^{-s}.
    \end{aligned}
    \end{eqnarray}
    Here $s:=s_1+s_2$, where $s_1:=\linte{\log_{1/\epsilon}(2N)}+1$ and $s_2$ satisfies that
    \begin{equation}\label{defofs2AB}
        \gamma^{2s_2-1}<(2C_1C_2\gamma^{s_1})^{-1}.
    \end{equation}
  Such integer exists because of  Definition \ref{defofpiee} (iii).
\medskip

\item For any integer $k\geq1$, let $\rho_k$ be the diameter of $B_k$ that Bob chooses at $k$-th step.\medskip

\item $n_0:=1$ and $n_k:=\inf\left\{n\in\N:\rho_{n}<\gamma^{ks}\rho_1\right\}$ for any integer $k\geq1$.
\medskip

\item $\delta:=\frac{\rho_1}{2}\left(\frac{1}{2C_2\gamma^{s_2}-1}-C_1\gamma^s\right)$. By \eqref{defofs2AB}, we have $\delta>0$.
\medskip

\item For any integer $k\geq1$, denote
\[
    \cI_k:=\bigcup_{N=1}^{\infty}\left\{\bfu\in\Sigma(1,N):  \gamma^{ks}< |I_{\bfu}(1,N)|\leq\gamma^{(k-1)s}\right\}.\medskip
\]
\end{itemize}

%With the above notations in mind, we next show that Alice has a strategy  to ensure that the set $\bigcap_{k\geq1}B_k$ intersects $N(T,\cG,\delta)$, where the set $N(T,\cG,\delta)$ is defined as in \eqref{defofntgd}.
%If it is proved, then by the inclusion $N(T,\cG,\delta)\subseteq N(T,\cG)$, the set $N(T,\cG)$ is $\gamma$-absolute winning, which is as desired.
%\medskip

Without loss of generality,  we assume that $\rho_k\to 0$ $(k\to\infty)$ and $\rho_1<\gamma^{2s}$. Then $n_k<+\infty$ for any $k\geq1$ and $n_k\to\infty$ as $k\to\infty$. Under these assumptions,  we make the following claim:
\begin{claim}\label{claimforfull}
  In the $\gamma$-hyperplane absolute game on $[0,1]^d$,  Alice has a strategy to ensure that
    \begin{eqnarray}\label{bnksntg}
        B_{n_k+s}\cap \tN(\cT,\cG,\delta,|\bfu|)^c\cap \big(I_{\bfu}(1,|\bfu|)\times[0,1]^{d-1}\big)=\emptyset,\quad\forall\ \bfu\in\cI_{k+2}
    \end{eqnarray}
    for any integer $k\geq0$. Here, $|\bfu|$ denotes the length of the word $\bfu$.
\end{claim}

As demonstrated below, the claim implies that $\tN(\cT,\cG)$ is hyperplane absolute winning on $[0,1]^d$.  We first show that, under the claim, the set 
 \begin{equation}\label{ncfcgcupe}
     \tN(\cT,\cG)\cup (E\times[0,1]^{d-1})
 \end{equation}
 is $\gamma$-hyperplane absolute winning on $[0,1]^d$, where $$E:=\bigcup_{N=1}^{\infty}\bigcup_{\bfu\in\Sigma(1,N)}\partial I_{\bfu}(1,N).$$ In fact, let $\bfx\in \bigcap_{k=1}^{\infty}B_k$ and suppose that $x_1\notin E$. Then, for any sufficiently large $n\in\N$, there exists $k\in\N$ and $\bfu\in\cI_{k+2}$ such that $|\bfu|=n$ and $x_1\in I_{\bfu}(1,n)$.   It follows from  \eqref{bnksntg}  that $\bfx\in\tN(\cT,\cG,\delta,n)$. Since $n$ is arbitrary, we have $\bfx\in \tN(\cT,\cG)$. The upshot of the above is that the set \eqref{ncfcgcupe}  is $\gamma$-hyperplane absolute winning on $[0,1]^d$. By the arbitrariness of $\gamma\in(0,1/3)$, the set  is hyperplane absolute winning on $[0,1]^d$. Moreover, note that $E$ is a countable set and $[0,1]^d\setminus\big(\{y\}\times[0,1]^{d-1}\big)$ is hyperplane absolute winning on $[0,1]^d$ for any $y\in[0,1]$, it follows from Proposition \ref{proofhypdi} (iii) that the set $\tN(\cT,\cG)$ is hyperplane absolute winning on $[0,1]^d$.\medskip

To prepare for proving the claim, we start with the following result on the lengths of words in $\cI_k$.

\begin{lemma}\label{lenofik}
Let $N\in\N$ be defined as above, then for every integer $k\geq1$ and every pair of $\bfu,\bfv\in\cI_k$ with $I_{\bfv}(1,|\bfv|)\subseteq I_{\bfu}(1,|\bfu|)$, we have $0\leq|\bfv|-|\bfu|<N$.
\end{lemma}
\begin{proof}
 Note that for every $k\in\N$, $\bfu\in\cI_k$ and $\bfw\in\Sigma(|\bfu|+1,N)$, we have
\begin{eqnarray*}
    |I_{\bfu\bfw}(1,|\bfu|+N)|&\leq &C_3\,|I_{\bfu}(1,|\bfu|)|\cdot |I_{\bfw}(|\bfu|+1,N)|\quad(\text{by \eqref{diaeofiuiv}})\\
    &\leq&C_3\,\gamma^{(k-1)s}\cdot\frac{1}{\inf\{|\sfT_{i,N,1}'(x)|:i\in\N,x\in\bigcup_{\bfw'\in\Sigma(i,N)}I_{\bfw'}(i,N)\}}\\
    &<&\gamma^{ks}\quad(\text{by the first inequality in \eqref{defofN1}})\,.
\end{eqnarray*}
On the other hand, it is easily verified that if $I_{\bfv}(1,|\bfv|)\subseteq I_{\bfu}(1,|\bfu|)$, then $|\bfv|-|\bfu|\geq0$ and $\bfv=\bfu\bfw$ for some $\bfw\in\Sigma(|\bfu|+1,|\bfv|-|\bfu|)$. By combining the above information, we obtain the desired inequality $0\leq|\bfv|-|\bfu|<N$. The proof is complete.
\end{proof}

Now we are going to prove Claim \ref{claimforfull} by induction.
\medskip 

\noindent\emph{$\bullet$ Claim \ref{claimforfull} is true for $k=0$}: By Lemma \ref{lenofik} and the assumption that $\rho_1<\gamma^{2s}$, we have
\[
    \#\left\{\bfu\in\cI_2:I_{\bfu}(1,|\bfu|)\cap B_1\neq\emptyset\right\}\leq 2N.
\]
This together with the definition of $\delta$, Lemma \ref{lemmaconint}, \eqref{diamofiu} and \eqref{defofs2AB} implies that there exist $2N$ intervals $J_1,J_2,...,J_{2N}\subseteq[0,1]$ of length
\[
    \leq(2\delta+C_1\rho_1)C_2\gamma^{s}\leq\frac{\gamma^{s_1+1}}{2}\rho_1<\frac{\gamma^{s_1}}{2}|B_1|
\]
such that 
\begin{eqnarray}\label{uini2b1capntg}
    \bigcup_{\bfu\in\cI_2}B_1\cap \tN(\cT,\cG,\delta,|\bfu|)^c\cap \big(I_{\bfu}(1,|
    \bfu|)\times[0,1]^{d-1}\big)\ \subseteq \ \bigcup_{i=1}^{2N}\left\{\bfx\in\R^d:x_1\in J_i\right\}.
\end{eqnarray}
On the other hand,  on applying Corollary \ref{corokeystra} (ii) to the  intervals $J_1,J_2,...,J_{2N}$, Alice has a strategy to ensure that 
\[
    B_{1+s_1}\cap\bigcup_{i=1}^{2N}\left\{\bfx\in\R^d:x_1\in J_i\right\}=\emptyset
\]
whatever Bob chooses $B_{1+s_1}$ at the $(1+s_1)$-th step. By \eqref{uini2b1capntg}, this implies that the ball $B_{1+s}\subseteq B_{1+s_1}$ is disjoint with $\tN(\cT,\cG,\delta,|\bfu|)^c\cap\big(I_{\bfu}(1,|\bfu|)\times[0,1]^{d-1}\big)$ for any $\bfu\in\cI_2$. This proves Claim \ref{claimforfull} for $k=0$.
\medskip

\noindent\emph{$\bullet$ If Claim \ref{claimforfull} is true for  $k=j$, then it is true for $k=j+1$}: The method of this step is analogous to that used in the case $k=0$. From inductive hypothesis, Bob has chosen $B_{n_j+s}$ that satisfies \eqref{bnksntg} with $k=j$.  To begin with, let Alice play arbitrarily until Bob selects $B_{n_{j+1}}$ (note that $n_{j+1}\geq n_{j}+s$ by definition). Then, on combining  the definition of $n_{j+1}$, the assumption $\rho_1<\gamma^{2s}$ with Lemma \ref{lenofik}, we have
\[
 \#\left\{\bfu\in\cI_{j+3}:I_{\bfu}(1,|\bfu|)\cap B_{n_{j+1}}\neq\emptyset\right\}\leq 2N.
\]
With this in mind, by the definition of $\delta$, Lemma \ref{lemmaconint}, \eqref{diamofiu} and \eqref{defofs2AB}, we obtain that there exist $2N$ intervals $J_1,J_2,...,J_{2N}\subseteq[0,1]$ with length
\[
    \leq(2\delta+C_1\gamma^{(j+1)s}\rho_1)\cdot C_2\gamma^{(j+2)s}\leq\frac{1}{2}\gamma^{(j+1)s+s_1+1}\rho_1<\frac{\gamma^{s_1}}{2}|B_{n_{j+1}}|
\]
such that 
\begin{eqnarray}\label{bnj1ind}
    \bigcup_{\bfu\in\cI_{j+3}}B_{n_{j+1}}\cap \tN(\cT,\cG,\delta,|\bfu|)^c\cap I_{\bfu}\subseteq\bigcup_{i=1}^{2N}\left\{\bfx\in\R^d:x_1\in J_i\right\}.
\end{eqnarray} 
Now, by applying Corollary \ref{corokeystra} (ii) to $J_1,...,J_{2N}$, Alice can find a strategy such that $B_{n_{j+1}+s_1}$ is disjoint with the set
\[
    \bigcup_{i=1}^{2N}\left\{\bfx\in\R^d:x_1\in J_i\right\}
\]
 whatever Bob chooses $B_{n_{j+1}+s_1}$ at the $(n_{j+1}+s_1)$-th step. On combining this and \eqref{bnj1ind}, we have
  \begin{eqnarray*}
        B_{n_{j+1}+s}\cap \tN(\cT,\cG,\delta,|\bfu|)^c\cap \big(I_{\bfu}(1,|\bfu|)\times[0,1]^{d-1}\big)=\emptyset,\quad\forall\ \bfu\in\cI_{j+3}.
    \end{eqnarray*}
 Therefore,   Claim \ref{claimforfull} is true for $k=j+1$.
\medskip

 The above proved Claim \ref{claimforfull}. As mentioned below the claim, this implies that the set $\tN(\cT,\cG)$ is hyperplane absolute winning on $[0,1]^d$, which completes the proof .

\bigskip

\noindent{\it Acknowledgments}: The author BL was supported by National Key R\&D Program of China (No. 2024YFA1013700), NSFC 12271176 and Guangdong Natural Science Foundation 2024A1515010946. The author Na Yuan was supported by Guangdong University
Young Innovative Talents Program Project (No. 2025KQNCX059).

\bibliographystyle{plain}
	\bibliography{Main.bib}
\end{document}